\documentclass[reqno]{amsart}
\usepackage{amssymb,amsmath,amsthm,bm,mathrsfs,paralist}
\usepackage{url,xcolor,units,graphicx,mathtools,url,amsrefs}
\usepackage[T1]{fontenc}
\RequirePackage[colorlinks,citecolor=blue,urlcolor=blue]{hyperref}
\definecolor{CYL}{rgb}{0.3,1,1}
\definecolor{DK}{rgb}{1,0.4,0.3}
\definecolor{FP}{rgb}{0,0.3,1}
\definecolor{YX}{rgb}{0.3,0.6,0}
\usepackage[cal=dutchcal,calscaled=1.05,
	scr=boondoxo,scrscaled=1.05,
	bb=boondox,bbscaled=1
	]{%
mathalfa}
\DeclareMathOperator{\Var}{Var}
\DeclareMathOperator{\Cov}{Cov}

\DeclareMathOperator{\lip}{{\rm Lip}}
\DeclareMathOperator{\dimh}{dim_{_H}\!}

\DeclareMathOperator{\ST}{\mathcal{X}}

\renewcommand{\>}{\rangle}
\newcommand{\N}{\mathbb{N}}
\newcommand{\Z}{\mathbb{Z}}

\newcommand{\R}{\mathbb{R}}
\newcommand{\C}{\mathbb{C}}
\renewcommand{\P}{\mathrm{P}}
\newcommand{\rQ}{\mathrm{Q}}
\newcommand{\E}{\mathrm{E}}
\newcommand{\F}{\mathcal{F}}
\newcommand{\sF}{\mathscr{F}}
\newcommand{\T}{\mathbb{T}}

\newcommand{\1}{\mathbb{1}}
\renewcommand{\d}{{\rm d}}

\newcommand{\e}{{\rm e}}
\renewcommand{\geq}{\geqslant}
\renewcommand{\leq}{\leqslant}
\renewcommand{\ge}{\geqslant}
\renewcommand{\le}{\leqslant}
\newtheorem{proposition}{Proposition}

\newtheorem{theorem}[proposition]{Theorem}
\newtheorem{lemma}[proposition]{Lemma}
\theoremstyle{definition} 
\newtheorem{definition}[proposition]{Definition}
\newtheorem{assump}[proposition]{Assumption}

\numberwithin{equation}{section}
\newtheorem{conjecture}[proposition]{Conjecture}
\numberwithin{equation}{section}
\numberwithin{proposition}{section}
\title{Uniform dimension theorems for parabolic SPDEs}
\thanks{Research supported in part by the NSF
        grants DMS-1855439 and DMS-2153846, National
        Natural Science Foundation of China (No.~12571153
        and No.~12201047), and the Shenzhen Peacock grant 2025TC0013.}
\author{Davar Khoshnevisan}
\address{Department of Mathematics, The University of Utah,
	Salt Lake City, UT 84112-0090 USA}
\email{davar@math.utah.edu}
\author{Cheuk Yin Lee}
\address{School of Science and Engineering, 
	The Chinese University of Hong Kong (Shenzhen), Longgang, Shenzhen, Guangdong, 518172, China}
\email{leecheukyin@cuhk.edu.cn}
\author{Fei Pu}
\address{Laboratory of Mathematics and Complex Systems, 
	School of Mathematical Sciences, Beijing Normal University, 100875, Beijing, China}
\email{fei.pu@bnu.edu.cn}
\author{Yimin Xiao}
\address{Department of Statistics and Probability, 
	Michigan State University, East Lansing, MI 48824, USA}
\email{xiao@stt.msu.edu}
\date{Last update: November 5, 2025}
\let\oldtocsection=\tocsection
\let\oldtocsubsection=\tocsubsection

\renewcommand{\tocsection}[2]{\hspace{0em}\oldtocsection{#1}{#2}}
\renewcommand{\tocsubsection}[2]{\hspace{2.5em}\oldtocsubsection{#1}{#2}}
\begin{document}
\begin{abstract} 
	Consider the following $p$-dimensional system of It\^o type stochastic PDEs, 
    \begin{align*}\left[\begin{aligned}
        &\partial_t u(t\,,x) = \partial^2_x u(t\,,x) + b(u(t\,,x))
            + \sigma(u(t\,,x)) \xi(t\,,x)\\
        &\text{for $(t\,,x)\in(0\,,\infty)\times\T$, subject to $u(0) \equiv u_0$
            on $\T$},
  \end{aligned}\right.\end{align*}
	where $\T$ denotes a given one-dimensional  torus, the initial 
	data $u_0:\T\to\R^p$ is assumed to 
	be fixed and non-random and in $C^{1/2}(\T\,;\R^p)$, 
	and $\xi$ denotes a $p$-dimensional
	space-time white noise. Under certain regularity conditions 
	on $b$ and $\sigma$, it is proved that, 
	if $p \ge 4$,  then 
		\begin{equation*}
			\P\{\dimh  u(\{t\}\times F) = 2\dimh  F \
			\text{$\forall$compact  $F\subset\T$, $t>0$}\}=1.
		\end{equation*}
	If in addition the matrix $\sigma(v)$ does not depend on $v\in\R^p$, and is nonsingular, then 
	the above equality holds for all $p\ge2$.
\end{abstract}
\maketitle



\section{Introduction}
Throughout this paper, we choose and fix a positive integer $p$,
and consider the following $p$-dimensional system of It\^o stochastic PDEs,\footnote{To
	be concrete, we mention that, in 
	\eqref{SHE}, the quantity $\sigma(u(t\,,x))\xi(t\,,x)$ is understood as an
	 It\^o type random-matrix product whose $i$th coordinate can be written explicitly as
	$\sum_{j=1}^p \sigma_{i,j}(u(t\,,x))\xi_j(t\,,x)$.
	}
\begin{equation}\label{SHE}\left[\begin{aligned}
	&\partial_t u(t\,,x) = \partial^2_xu(t\,,x) + b(u(t\,,x))
		+ \sigma(u(t\,,x)) \xi(t\,,x)\\
	&\text{for $(t\,,x)\in(0\,,\infty)\times\T$, subject to $u(0) \equiv u_0$
		on $\T$},
\end{aligned}\right.\end{equation}
where $\T$ denotes a given one-dimensional  torus
and the initial data $u_0:\T\to\R^p$ is assumed to be fixed and non-random and
in $C^{1/2}(\T\,;\R^p)$.\footnote{As is customary, for $\alpha \in (0, 1]$, $C^\alpha(\T\,;\R^p)$ denotes
	the linear space of all functions $f:\T\to\R^p$ such that
	$\sup_{a,b\in\T:a\neq b}\|f(b)-f(a)\|/|b-a|^\alpha<\infty$.
	}
The random forcing term $\xi$ denotes a $p$-dimensional
space-time white noise, equivalently, 1-D white noise on 
$\{1,\ldots,p\}\times\R_+\times\T$. In brief terms, this means that $\xi$ is a 
centered, generalized Gaussian random field such that
\[
	\Cov[ \xi_i(t\,,x)\,, \xi_j(s\,,y)] = \delta_0(i-j)\delta_0(t-s)\delta_0(x-y),
\]
for all $i,j\in\{1,\ldots,p\}$, $t,s\ge0$, and $x,y\in\T.$
Finally, $b:\R^p\to\R^p$ and $\sigma:\R^p\to\R^p\times\R^p$
are assumed to be Lipschitz continuous.
In this way, standard methods such as those in Walsh \cite{W}
can be employed to show that \eqref{SHE} is well posed
and the solution $(t\,,x)\mapsto u(t\,,x)$ is almost surely 
in $C^{a,b}(\R_+\times\T\,;\R^p)$ for every $a\in(0\,,\frac14)$ and $b\in(0\,,\frac12)$.

Equations such as \eqref{SHE} arise in the analysis of reaction-diffusion systems,
or multi-component reaction-diffusion equations; see Fife
\cite{Fife} for a masterly account. In this context, $b$ denotes the sink/source terms,
$p=$ the number of components (i.e., the number of interacting reactions), $\xi=$
the external forcing term, and $\sigma$ encodes the interactions between
the components and the forcing. It has been known for a long time that when $\xi$ is
a nice and smooth function, the system \eqref{SHE} can act phenomenologically
differently for higher values of $p$ than it would for instance when $p=1$. For an early
example see Turing \cite{Turing}; more modern examples abound in the literature on 
bifurcation theory. 

On one hand, the results of the present paper demonstrate that when $\xi$ is white noise,
some of the fractal  behavior of the solution has the same type
of phenomenological dimension dependence as one sees in the deterministic theory. 
On the other hand, we will see that, in stark contrast with the deterministic theory,
dimension dependence can arise solely because the
system ``suffers from too much noise'' when $p$ is large,
essentially, regardless of the details
of the construction of $b$ and $\sigma$. Before we explain our results more precisely,
let us discuss a little of the requisite background.

The following is due to Dalang, Khoshnevisan, and Nualart \cites{DKN1,DKN2},
except that they consider equations with Neumann boundary instead of the
present setting which can be viewed as an equation with a periodic 
boundary. The next result can be proved in almost exactly the same manner
as the results of Dalang et al. ({\it ibid.}). As usual, $\dimh$ denotes the Hausdorff
dimension.

\begin{proposition}\label{pr:DKN}
	Suppose, in addition, that $\sigma$ and $b$ are uniformly bounded and $C^\infty$ and
	$\sigma$ is uniformly elliptic, then 
	$\P\{\dimh u(\{t\}\times\T)=2\}=1$ for all $t>0$ when $p\ge 3$. If $\sigma(v)$
	does not depend on $v\in\R^p$ (additive noise), then in fact
	$\P\{\dimh u(\{t\}\times\T)=2\}=1$ for all $t>0$ when $p\ge 2$. 
\end{proposition}

One can anticipate Proposition \ref{pr:DKN} using the following heuristic:
Recall that the random function 
$x\mapsto u(t\,,x)$ is a ``smooth perturbation'' of a Brownian motion for each fixed $t>0$.
This can be proved by
combining the localization results of Foondun et al \cite{FKM} with the structure theory
in \cite{DK}*{\S3}. Since Hausdorff dimension is a local quantity, one would then
expect that $\dimh u(\{t\}\times\T)$ ought to be the same as $\dimh B(\T)$ when
$\{B(x)\}_{x\in\T}$ (say) denotes a two-sided,
$p$-dimensional Brownian motion indexed by $\T\cong[-1\,,1]$,
where $B(0)=0$. Thus, we can anticipate Proposition \ref{pr:DKN} since a celebrated theorem of
Paul L\'evy asserts that $\dimh B(\T)=2$ almost surely. If we use the same heuristic but
replace L\'evy's theorem with the theorem of McKean \cite{McKean}, then we might 
expect that if, in addition, $\sigma$ and $b$ are uniformly bounded and $C^\infty$ and
$\sigma$ is uniformly elliptic, then 
\begin{equation}\label{pre:dim}
	\P\{\dimh u(\{t\}\times F)=2\dimh F\}=1\
	\ \forall\text{compact}\ F\subset\T,\ t>0,
\end{equation}
when $p\ge3$. And that, if $\sigma(v)$
does not depend on $v\in\R^p$ (additive noise), then in fact
\eqref{pre:dim} holds for all $p\ge 2$. Because this is an assertion about the local behavior
of the solution to \eqref{SHE}, it should be possible to refine the method of
proof of Proposition \ref{pr:DKN} in order to prove such a result. Instead, we 
plan to prove the following Theorem \ref{th:HEAT:torus}, though it has a very different proof.
It might help to recall that, for a closed non-random set $A\subset\R^p$,
\[
	A\text{ is \emph{polar} if }\P\{\exists (t\,,x)\in(0\,,\infty)\times\T:
	u(t\,,x)\in A\}=0.
\]

\begin{theorem}\label{th:HEAT:torus}
	Suppose that the underlying probability space $(\Omega\,,\mathscr{F},\P)$ is complete.
	If $\{v\in\R^p:\, \inf_{x\in\R^:\|x\|=1}
	\|\sigma(v)x\|=0\}$ is polar for the random field $u$ 
	(Assumption \ref{ass:polar}) and $p \ge 4$,  then 
	\begin{equation}\label{eq:HEAT:torus}
		\P\{\dimh  u(\{t\}\times F) = 2\dimh  F \
		\text{$\forall$compact  $F\subset\T$, $t>0$}\}=1.
	\end{equation}
	If in addition the matrix $\sigma(v)$ does not depend on $v\in\R^p$, and is nonsingular, then 
	\eqref{eq:HEAT:torus} holds for all $p\ge2$.
\end{theorem}

Theorem \ref{th:HEAT:torus} presents a nontrivial extension of \eqref{pre:dim}:
First of all, the polarity condition of Theorem \ref{th:HEAT:torus} is essentially unimproveable
and  includes but is not limited to strongly elliptic $\sigma$. More significantly, the 
dimension formula \eqref{eq:HEAT:torus}
is valid, off a single null set, simultaneously for all compacts $F$ and times $t>0$,
including random ones. In a sense, Theorem \ref{th:HEAT:torus} can be anticipated (at least for a
fixed $t>0$) from the celebrated uniform dimension
theorem of Kaufman \cite{Kaufman} which states
that
\[
	\P\{\dimh B(F)=2\dimh F\ \forall\text{compact } F\subset\T\}=1
	\quad\text{when $p\ge2$,}
\]
where $B$ once again denotes a two-sided, $p$-dimensional Brownian motion indexed by $\T$.
However, this heuristic comparison to Brownian motion
does not appear to be rigorizable since, in contrast for example
with \eqref{pre:dim}, uniform dimension results cannot be based solely on local arguments. 

In Conjecture \ref{conj:add} below,
we will state a series of uniform dimension conjectures of a similar nature
that we have no idea how to prove except in the additive case. One of them can be stated
right here, as it pertains directly to the material of the Introduction.

\begin{conjecture}\label{conj1}
	The uniform dimension theorem \eqref{eq:HEAT:torus} is valid whenever $p \ge 2$ provided that
	the set $\{v \in \R^p:\, \inf_{\|x\|=1}
	\|\sigma(v)x\|=0\}$ is polar for $u$.
\end{conjecture}

Suppose that $p=1$,
and consider the random compact set
\[
	F_1 = \{x \in \T: u(1\,,x) = 0\}.
\]
If in addition $b,\sigma$ are bounded and smooth and $\sigma$ is strongly elliptic, then
Corollary 1.7(d) of \cite{DKN2} implies that
$\P\{F_1 \neq\varnothing\}>0$, and $\dimh F_1 = 1/2$ a.s.~on
$\{F_1 \ne \varnothing\}$, whence $\P\{ \dimh u(1\,,F_1) = 0 \ne 1 = 2 \dimh F_1 \}>0$.
This yields the optimality of the preceding conjecture. And of course,
the polarity condition of Conjecture \ref{conj1} is also sharp. 
For instance, suppose that $\sigma(c)=0$
for some $c\in\R^p$ and $u_0\equiv c$. Then, $u(t\,,x)\equiv c$ for all 
$t$ and $x$ and so \eqref{eq:HEAT:torus} fails manifestly.

We include an outline of the proof of Theorem \ref{th:HEAT:torus} in \S2, together with
statements of more detailed results about the SPDE
\eqref{SHE} in
the case that it is driven by additive noise. 
The remainder of the paper is devoted to the proof of these results.
Let us conclude the Introduction by setting forth a long series of notational conventions that
will be used throughout the paper.

For every $m\in\N$ and $x\in\R^m$, the point $x$ is written coordinatewise as 
$x=(x_1,\ldots,x_m),$ 
and similarly, the $i$th coordinate of every function $f:\R^n\to\R^m$ is written as
$f_i$ and the $(i\,,j)$th coordinate of a matrix $M$ is written as $M_{i,j}$.
For every $n\times m$ matrix $M$ we let $\|M\|=(\sum_{i,j} M_{i,j}^2)^{1/2}$ 
denote the Hilbert--Schmidt norm of $M$, so that $\|Mx\|\le \|M\|\|x\|$ for every $x\in\R^m$.
We also write 
\[
	\log_+ x = \log(x\vee\e)\qquad\forall x\ge0.
\]

Throughout, $\T$ denotes the set $[-1\,,1]$ endowed
with addition mod 2 so that $\T\cong\R/2\Z$, as an abelian group.
We also identify $\T$ with the circle group $\mathbb{S}=\{x\in\C:\,\|x\|=1\}$,
endowed with multiplication on $\C$, using the
group homomorphism $h:x\mapsto \exp(i\pi x)$.
The mapping $h$ is a 1-1 isometry when 
we view $\mathbb{S}$ as a manifold with Riemannian distance and yields
$\text{dist}(a\,,b)= |a-b|\wedge (2-|b-a|)$ for $a,b\in\T$. Throughout, we
follow standard practice and use additive notation for $\T$. Thus,
 ``$+$'', ``$-$'', and ``$0$'' respectively denote the group multiplication,
 inversion, and identity. Similarly, we write $|a|=\text{dist}(a\,,0)$ for all $a\in\T$.
 We write ``$g_1(x)\lesssim g_2(x)$ for all $x\in X$'' when
there exists a positive real number $L$ such that $g_1(x)\le Lg_2(x)$ for all $x\in X$.
Alternatively, we might write ``$g_2(x)\gtrsim g_1(x)$ for all $x\in X$.'' By
``$g_1(x)\asymp g_2(x)$ for all $x\in X$'' we mean that $g_1(x)\lesssim g_2(x)$ and $g_2(x)\lesssim g_1(x)$ for all $x\in X$.

If $k\in[1\,,\infty)$ is a real number and $Y$ is a random $n\times m$ matrix, then we write
$\|Y\|_k = \E(\|Y\|^k)^{1/k}$
regardless of the values of $n,m\in\N$.
If $\Phi=\{\Phi(x)\}_{x\in\T}$ is a $\T$-indexed random field with values in $\R^p$,
then for every real number $k\in[2\,,\infty)$ we write
\begin{equation}\label{S:H}
	\mathcal{S}_k(\Phi) = \sup_{a\in\T}\|\Phi(a)\|_k
	\quad\text{and}\quad
	\mathcal{H}_k(\Phi) = \sup_{\substack{a,b\in\T\\a\neq b}}
	\frac{\|\Phi(b)-\Phi(a)\|_k}{|b-a|^{1/2}}.
\end{equation}
For every function $f:\R^p\to\R^m$, we write
\begin{equation}\label{Lip}
	\lip(f) = \sup_{\substack{a,b\in\R^p\\a\neq b}}\frac{\|f(b)-f(a)\|}{\|b-a\|}
	\quad\text{and}\quad
	\mathcal{M}(f) = \sup_{v\in\R^p}\|f(v)\|.
\end{equation}

The Fourier transform on $\R$ is denoted by ``$\hat{\phantom{w}}$''
and is normalized so that 
\begin{align}\label{E:R_FT}
	\hat{f}(\xi) = \int_{-\infty}^\infty \e^{-ix\xi} f(x)\, \d x \quad \forall\xi \in \R,
	\ f\in L^1(\R).
\end{align}
We will denote the Fourier transform on $\T$ by $\F$, in order to 
distinguish it from the Fourier transform on $\R$,
and normalize it so that
\begin{align}\label{E:T_FT}
	\F f(n) = \int_{-1}^1 \e^{-\pi i n z} f(z) \,\d z
	\quad\text{and}\quad
	\F^{-1}g(z) = \tfrac12 \sum_{n=-\infty}^\infty \e^{i\pi n z} g(n),
\end{align}
for every $n\in\N$, $f\in L^1(\T)$, $g\in L^1(\Z)$, and $z\in\T$.

The open  and closed balls in $\R^p$ centered at 
$x \in \R^p$ with radius $r>0$ are respectively denoted  by
\begin{align}\label{E:B}
	\mathbb{B}(x\,,r) = \left\{ y\in\R^p:\, \|x-y\| < r\right\}
	\quad\text{and}\quad
	B(x\,,r) = \overline{\mathbb{B}(x\,,r)}.
\end{align}
Throughout, we frequently refer to the set,
\begin{equation}\label{ST}
	\ST= \R_+\times\T,
\end{equation}
as \emph{space-time}, and
view it as a metric space that is endowed with the following so-called \emph{parabolic metric} $\rho$:
\begin{equation}\label{def:rho}
	\rho((s\,,y)\,,(t\,,x)) = |s-t|^{1/4} + |x-y|^{1/2}
	\quad\forall (s\,,y),(t\,,x)\in\ST.
\end{equation}
The corresponding open and closed balls are denoted, respectively, by
\[
	\mathbb{B}_\rho(a\,,r) = \{ b \in \ST: \rho(a\,,b) < r \},\quad
	B_\rho(a\,,r) = \left\{ b \in \ST: \rho(a\,,b) \le r \right\},
\]
whenever $a\in\ST$ and $r>0$.

There are associated notions of Hausdorff measure
and Hausdorff dimension
on space-time $\ST$: Whenever $E\subset \ST$ and
$\beta \ge 0$, the \emph{$\beta$-dimensional Hausdorff measure}
of $E$ with respect to the metric $\rho$ is defined by
\[
	\mathcal{H}^\rho_\beta(E)=\lim_{\delta\to0}\ \inf
	\sum_{n=1}^\infty (2r_n)^\beta,
\]
where the inf is taken over all countable open covers $\mathbb{B}_\rho(a_1,r_1),
\mathbb{B}_\rho(a_2\,,r_2),\ldots$ of $E$ such that $\sup_{n\ge1}r_n\le\delta$.
The \emph{Hausdorff dimension} of $E$ with respect to $\rho$ is defined by
\begin{align}\label{E:dim:rho}
	\dim_{_{\rm H}}^\rho E = \inf\{ \beta \ge 0 : \mathcal{H}^\rho_\beta(E) =0 \}
	=\sup\{\beta\ge0: \mathcal{H}^\rho_\beta(E)=\infty\}.
\end{align}
It helps to have a slightly broader definition of Hausdorff dimension
for some of our later purposes. Suppose $N\in\N$ is fixed and
$\mathsf{d}$ is a metric on $\R^N$. Then, we can define the Hausdorff dimension
$\dim^{\mathsf{d}}_{_{\rm H}}$ and associated Hausdorff measure
$\mathcal{H}_\beta^{\mathsf{d}}$ on the metric space $(\R^N\!\!,\mathsf{d})$ as
above, but replace all references to $\rho$ there by their counterparts 
that use the metric $\mathsf{d}$ here. 

Throughout, $\{\mathscr{F}_t\}_{t\ge0}$ denotes the filtration generated by
the Gaussian noise $\xi$. That is, $\mathscr{F}_t$
denotes the $\sigma$-algebra generated by all Wiener integrals
of the form $\int_{[0,t]\times\T}f(y)\cdot\xi(\d s\,\d y)$ as $(t\,,f)$
ranges over $(0\,,\infty)\times L^2(\T\,;\R^p)$. Without loss of generality, we may
and will assume that $\{\mathscr{F}_t\}_{t\ge0}$ satisfies the usual
conditions of martingale theory. In particular, every $\mathscr{F}_t$
is $\P$-complete. We assume also that the underlying probability space
$(\Omega\,,\mathscr{F},\P)$ is complete. These completeness
assumptions are critical to ensure the measurability
of various sets that are of interest.

\section{A proof outline, the additive-noise case, and some conjectures}
The nonlinear part of
Theorem \ref{th:HEAT:torus} naturally includes two
very different statements. The first is that, off a single $\P$-null set,
\begin{equation}\label{E1}
	\dimh u(\{t\}\times F)\le 2\dimh F\qquad\forall\text{compact }F\subset\T,\
	t>0.
\end{equation}
This statement in fact holds for all $p\ge 1$, does not require
the polarity assumption of Theorem \ref{th:HEAT:torus}, and  is an immediate
consequence of the definition of Hausdorff dimension and the following
version of the well-known modulus of continuity of $u$: 
For every $\alpha\in(0\,,\frac12)$ and $T>0$, 
\begin{equation}\label{E2}
	\sup_{t\in[0,T]}\sup_{x,y\in\T:x\neq y}
	\frac{\|u(t\,,x) - u(t\,,y)\|}{|x-y|^\alpha}
	<\infty\quad\text{a.s.}
\end{equation}
See Walsh \cite{W}*{Chapter 3} for very closely related results with essentially
the same proofs.
Moreover, the preceding argument implies also that the event defined in
\eqref{E1} includes the event defined in \eqref{E2}. Since the probability
space is complete and the event of \eqref{E2} is measurable, then so is 
the event defined by \eqref{E1}.
Therefore, measurability issues do not arise in this first step. 

One can readily fill in the few remaining gaps to see that the preceding outline yields
a complete proof of \eqref{E1}, modulo a simple and well-known 
covering argument,
such as Proposition 2.3 of Falconer \cite{Fa},
and holds without need for the additional
technical assumptions of Theorem \ref{th:HEAT:torus}. 
We will concentrate our efforts on proving the second implication of
Theorem \ref{th:HEAT:torus}. Namely
that, if the set $\{v \in \R^p :\inf_{\|x\|=1}\|\sigma(v)x\|=0\}$
is polar for $u$ and $p\ge 4$,  then
\begin{equation}\label{E3}
	\dimh u(\{t\}\times F)\ge 2\dimh F\qquad\forall\text{compact }F\subset\T,\
	t>0,
\end{equation}
off a single $\P$-null set.
Because $\dimh$ is countably stable \cite{Fa},
it suffices to prove that \eqref{E3} holds simultaneously
for all $t\in[a\,,b]$ and $F\subseteq B(0\,,c)$
where $0<a<b$ and $c>0$ are non-random and fixed.
Then, one can try to emulate the method of Kaufman \cite{Kaufman},
introduced first in the context of Brownian motion. Namely, choose and fix
some $\alpha\in(0\,,2)$, and for every $n\in\N$ 
cover $B(0\,,c)$ as optimally as possible with finitely many balls $B(a_1\,,2^{-\alpha n}),
B(a_2\,,2^{-\alpha n}),\ldots,$ and then prove that simultaneously for every subset 
$a_{i_1},\ldots,a_{i_N}$ of the $a_j$s, and for all $t\in[a\,,b]$, the number of
inverse images of the form 
$\{x\in\T: u(t\,,x)\in B(a_j\,,2^{-\alpha n})\}$ [$1\le j\le N$] that intersect 
any dyadic arc in $\T$ grows at most polylogarithmically in $2^{n}$. Since the underlying
probability space is complete, this takes care of all measurability issues as well. 

In the context of Kaufman's theorem \cite{Kaufman}, the latter intersection estimate can be carried
out readily thanks to the fact that Brownian motion has stationary and independent
Gaussian increments. In the present setting, the required estimates are very
difficult to develop in part because of the inherent nonlinear dependence of $u$ on the underlying
noise. Thus, as a first step, we study the special case of \eqref{SHE} where it is driven by
a non-degenerate additive noise. In that case, a standard stopping argument reduces the problem
to the one where $b$ is uniformly bounded. Then, an appeal to the Girsanov theorem 
allows us to reduce the problem to the case that $b\equiv0$ when the solution is now a Gaussian
random field. Moreover, barring the relatively easy-to-manage effects of the initial data,
that Gaussian random field has all but one requisite property of Brownian motion: It only does not
have independent increments! While this is a significant loss of information,
it was shown by Wu and Xiao \cite{WuXiao} and Xiao \cite{X09} that
one can appeal to the theory of strong 
local non-determinism (SLND) in order to at least partially overcome the lack of independence;
see also Berman \cite{B73}, Lee and Xiao \cite{LX}, and Monrad and Pitt \cite{MP87}.
Thus, the main impediment to a proof of the following is to develop a suitable notion
of SLND. As it turns out, there are a number of variations of the notion of 
SLND that are available here. Once we identify the senses in which
the process $u$ has SLND, we are led to the three parts of the next
theorem, each part corresponding to a different notion of SLND. Note that 
part \emph{(i)} of Theorem \ref{th:HEAT:add} below shows
the definitive form of Theorem \ref{th:HEAT:torus},
with optimal conditions, valid in the case of
additive noise. Now recall \eqref{ST} and \eqref{def:rho}.

\begin{theorem}[Additive noise]\label{th:HEAT:add}
	If $\sigma$ is a constant nonsingular matrix, then:
	\begin{compactenum}[(i)]
	\item \(
			\P\{ \dimh  u(\{t\}\times A) = 2\dimh  A  \text{ $\forall$compact 
			$A \subset \T, t>0$} \}=1
		\) $\forall p\ge2$;
	\item \(
			\P\{ \dimh  u(B\times\{x\}) = 4\dimh  B \text{
			$\forall$compact  $B\subset (0\,, \infty), x \in \T$} \}=1
		\) $\forall p \ge 4$;
	\item \(
			\P\{\dimh  u(C) = \dim_{_{\rm H}}^\rho C \text{
			$\forall$compact 
			$C \subset\ST$}\}=1
		\) $\forall p \ge 6$.
	\end{compactenum}
\end{theorem}

Next, we consider \eqref{SHE} in the multiplicative case that
$\sigma: \R^p \to \R^p\times\R^p$ is
not necessarily a constant matrix, and discuss how we prove a
weaker form of \eqref{E3} in which $t>0$ is fixed (so that the null set
off which \eqref{E3} can depend on $t$). The proof of the uniform-in-$t$ result
is more complicated, and discussed in detail in the forthcoming arguments.
Barring the uniformity issue in $t$, this is the general setting
under which Theorem \ref{th:HEAT:torus} is posed. 

We had mentioned that locally $x\mapsto u(t\,,x)$ is approximately
a Brownian motion. Instead of that approximation, our proof hinges
on the fact that locally $t\mapsto u(t\,,x)$ is approximately 
a fractional Brownian motion (fBm) of index $\frac14$; see the unpublished
manuscript by Khoshnevisan, Swanson, Xiao, and Zhang \cite{KSXZ}.
We use this approximation theorem by conditioning on everything up to
time $t-\eta$ for a suitable choice of
$\eta=\eta(t)\in(0\,,t)$. If $\eta$ is sufficiently 
close to $t$, then one can approximate $[t-\eta\,,t]\ni s\mapsto u(s)$
by a Gaussian process that resembles a scaled version of an fBm with index
$\frac14$; the scaled version arises since the variance
process of the Gaussian approximation to $[t-\eta\,,t]\ni s\mapsto u(s)$
is small when $\eta\approx t$. Pitt \cite{Pitt78} showed that fBm
has the SLND property. A similar argument shows that
our Gaussian approximation to $[t-\eta\,,t]\ni s\mapsto u(s)$ enjoys a similar 
property. This, and the Gaussian approximation itself, together yield the probability
estimates that are needed in order to carry out a combinatorial analysis of the 
intersection numbers of inverse images of the form
$\{x\in\T: u(t\,,x)\in B(a_j\,,2^{-\alpha n})\}$ [$1\le j\le N$], brought up earlier
in the context of equations with additive noise. And the condition $p\ge 4$,
which we believe is too strong, comes up in this part for technical reasons that roughly say
that the scaling imposed by replacing $[0\,,t]$ by $[t-\eta\,,t]$ -- and then conditioning on
everything by time $t-\eta$ -- does not do a great deal of harm, so that the approximating
Gaussian random field behaves essentially as does fBm with index $\frac14$. In this way,
the analysis of the inverse images resembles that of the inverse images of fBm. 
Suffices it to say only that the proof of \eqref{E3} that includes the uniformity in $t$
rests on a similar method but uses a much more delicate form of SLND that is valid on many fine
scales of $t$ simultaneously.

Let us conclude this section by introducing notation for the singular values of the matrix $\sigma(v)$.
Recall that the \emph{singular values} of a $p\times p$ matrix $M$ are
the eigenvalues of $M'M$. As such, singular values are nonnegative.

\begin{definition}\label{lambda}
	Throughout, let $\lambda(v)= \inf_{x\in\R^p\setminus\{0\}} \|\sigma(v)x\|^2/\|x\|^2$
	denote the smallest singular value of
	$\sigma(v)$ for every $v\in\R^p$. 
\end{definition}
Note that the polarity condition of Theorem \ref{th:HEAT:torus} can 
be restated, a little more succintly,
as follows.
\begin{assump}[Polarity]\label{ass:polar}
	 $\lambda^{-1}\{0\}$ is polar for $u$.
\end{assump}

We have no idea how to prove the following, but in light of Theorem \ref{th:HEAT:add},
and despite the nonlocal nature of uniform dimension results, we feel that the following
is likely to be true.

\begin{conjecture}\label{conj:add}
	Theorem \ref{th:HEAT:add}\emph{(ii)}
	and \emph{(iii)} 
	are valid in the presence of multiplicative noise provided that $\lambda^{-1}\{0\}$
	is polar for $u$.
\end{conjecture}

Let us conclude this section by presenting the next observation
which readily implies that
$\lambda$ is locally Lipschitz on the open set $\R^p\setminus\lambda^{-1}\{0\}$.

\begin{lemma}\label{lem:lambda:cont}
	$\lambda^{1/2}$ is Lipschitz continuous on $\R^p$.
\end{lemma}

This must be well known. But the proof is so short that it might be easier to simply present
the proof.

\begin{proof}
	For every $v\in\R^p$,
	by compactness there exists $x(v)\in\R^p$ such that $\|x(v)\|=1$ and
	$\|\sigma(v)x(v)\|^2=\lambda(v)$. Thus, 
	\[
		\sqrt{\lambda(v)} = \|\sigma(v)x(v)\| \ge \|\sigma(w)x(v)\| - \lip(\sigma)\|v-w\|
		\ge\sqrt{\lambda(w)} - \lip(\sigma)\|v-w\|,
	\]
	valid for every $v,w\in\R^p$. Reverse the roles of $v$ and $w$ to finish the proof.
\end{proof}

\section{The Gaussian case}

In this section, we study the specialization of the SPDE \eqref{SHE} to the case that
$\sigma(v)$ is equal to the $p\times p$ identity matrix $\bm{I}$ for all $v\in\R^p$,
and that $u_0\equiv0$. Because
of the important role of this specialization, we reserve the symbol $H$ --
for ``heat'' -- for the solution
to \eqref{SHE} in that case. In other words,
\begin{equation}\left[\begin{aligned}\label{H}
	&\partial_t H =  \partial^2_x H +   \xi\quad\text{on $(0\,,\infty)\times\T$},\\
	&\text{subject to}\quad H(0) = \mathbb{0},
\end{aligned}\right.\end{equation}
where $\mathbb{0}:\T\to\R^p$ takes the value $0\in\R^p$ everywhere on $\T$.
The unique (mild) solution to \eqref{H} is the following Wiener-integral process,
\begin{equation}\label{E:T_mild}
	H(0\,,x)=0\quad\text{and}\quad
	H(t\,,x) =  \int_{(0,t)\times\T} G_{t-s}(x\,,y)\, \xi(\d s\,\d y),
\end{equation}
for all $t>0$ and $x\in\T$, where 
\begin{equation}\label{G}
	G_r(a\,,b) = \frac{1}{\sqrt{4\pi r}}\sum_{n=-\infty}^\infty\exp\left(
	-\frac{(a-b+2n)^2}{4r} \right)
	\qquad[r>0,\, a,b\in\T]
\end{equation}
denotes the heat kernel for $\partial_t-\partial_x^2$ on 
$(0\,,\infty)\times\T$. See da Prato and Zabczyk \cite{DZ} and Walsh \cite{W}.
We pause to observe that the random function $H=(H_1\,,\ldots,H_p)$ is made of
$p$ i.i.d.\ coordinate processes $H_1,\ldots,H_p$.

Our analysis of the random field $H$ will rely on various heat-kernel estimates
for the heat kernel $G$ that are closely related to some of the technical
estimates in Khoshnevisan, Kim, Mueller, and Shiu
\cite{KKMS20}. For instance, Lemma B.1 and Remark B.2
of \cite{KKMS20} together assert that
\begin{equation}\label{G:le}
	\tfrac14\max( t^{-1/2},1)\le
	\sup_{x,y\in\T}G_t(x\,,y) \le 2\max( t^{-1/2},1)\qquad\forall t>0.
\end{equation}
Next we present three lemmas that tighten up some of the other heat-kernel estimates of \cite{KKMS20},
improving them to what we believe are their respective essentially-optimal forms. We have resisted 
the temptation of deriving matching lower bounds  as we will not need them in the sequel.

\begin{lemma}\label{lem:G-G:x}
	For every fixed $T>0$,
	\[
		\int_0^t\d s\int_\T\d y\left[ G_s(x\,,y) - G_s(z\,,y)\right]^2
		\lesssim\sqrt t\wedge|x-z|,
	\]
	uniformly for all $x,z\in\T$ and $t\in(0\,,T]$.
\end{lemma}

\begin{proof}
	We can reiterate \eqref{G} as follows: For all $a,b\in\T$, $v\in\R$, and $s>0$,
	\[
		G_s(a\,,b) = \sum_{n=-\infty}^\infty\phi_s(a-b+2n),
	\]
	where 
	\begin{equation}\label{G:phi}
			\phi_s(v) = \frac{\exp(-v^2/(4s))}{(4\pi s)^{1/2}}
	\end{equation}
	denotes the heat kernel for $\partial_t-\partial^2_x$ on $\R$ (vs.\ $\T$).
	We can use the semigroup property of the heat kernel in order to see that
	\begin{align*}
		&
			\int_\T\left[ G_s(x\,,y) - G_s(z\,,y)\right]^2\d y = 2G_{2s}(0\,,0)
			-2 G_{2s}(x\,,z)\\
		&
			=2\sum_{n=-\infty}^\infty \left[
			\phi_{2s}(2n)-\phi_{2s}(x-z+2n)\right]
			= \sum_{n=-\infty}^{\infty} \left(\e^{-2\pi^2 sn^2}- 
			\e^{\pi i (x-z)n-2\pi^2 sn^2}\right)\\
		&
			= 2\sum_{n=1}^{\infty} \e^{-2\pi^2 sn^2}\left(1- \cos(\pi  (x-z)n)\right),
	\end{align*}
	where we appealed to the Poisson summation 
	formula (Lemma \ref{lem:poisson}) in order to deduce the third identity.
	Thus,
	\begin{align*}
		&\int_0^t\d s\int_\T\d y \left[ G_s(x\,,y) - G_s(z\,,y)\right]^2
			\leq 2\sum_{n=1}^\infty \frac{1 - 
			\e^{-2\pi^2 t n^2}}{2\pi^2 n^2} \left(2\wedge (\pi^2 (x-z)^2n^2)\right)\\
		&\lesssim \left(\sum_{n=1}^\infty 
			\frac{1\wedge (tn^2)}{n^2}\right) \wedge \left(\sum_{n=1}^\infty 
			\frac{1\wedge ((x-z)^2n^2)}{n^2}\right) \lesssim \sqrt{t} \wedge |x-z|
			\quad\text{(Lemma \ref{inequality1})},
	\end{align*}
	uniformly for all $t>0$ and $x,z\in\T$. This does the job.
\end{proof}

\begin{lemma}\label{lem:G-G:t}
	For every $T>0$ fixed,
	\[
		\int_0^r\d s\int_\T\d y\, [ G_{t-s}(x\,,y)-G_{r-s}(x\,,y)]^2
		\lesssim\sqrt{t-r},
	\]
	uniformly for all $x\in\T$ and $0<r<t\leq T$.
\end{lemma}

\begin{proof}
	Thanks to the semigroup property of the heat kernel,
	\begin{gather*}
		\int_\T\left[ G_{t-r+s}(x\,,y)-G_{s}(x\,,y)\right]^2\d y 
		=  G_{2(t-r+s)}(0\,,0) + G_{2s}(0\,,0) - 2G_{t-r+2s}(0\,,0)\\
		= \sum_{n=-\infty}^\infty\phi_{2(t-r+s)}(2n) 
			+ \sum_{n=-\infty}^\infty\phi_{2s}(2n)
			-2\sum_{n=-\infty}^\infty\phi_{t-r+2s}(2n).
	\end{gather*}
	Then, by the Poisson summation formula (Lemma \ref{lem:poisson}), we have
	\begin{align*}
		&\int_\T\left[ G_{t-r+s}(x\,,y)-G_{s}(x\,,y)\right]^2\d y\\
		&=\frac{1}{2}\sum_{n=-\infty}^\infty \left(\e^{-2\pi^2(t-r+s)n^2}
			+\e^{-2\pi^2sn^2}-2\e^{-2\pi^2(t-r+2s)n^2}\right)\\
		&=\frac{1}{2} \sum_{n=-\infty}^{\infty}
			\left(\e^{-\pi^2(t-r+s)n^2} - \e^{-\pi^2sn^2}\right)^2
			=\sum_{n=1}^{\infty}\e^{-2\pi^2sn^2}\left(1- \e^{-\pi^2(t-r)n^2}\right)^2,
	\end{align*}
	for all $t>r>0$, $s\geq 0$ and $x\in\T$. Hence, a change of variable yields that
	\begin{align*}
	&\int_0^r\d s\int_\T\d y \left[ G_{t-s}(x\,,y)-G_{r-s}(x\,,y)\right]^2
		=\int_0^r\d s\, 
		\int_\T\left[ G_{t-r+s}(x\,,y)-G_{s}(x\,,y)\right]^2\d y\\
	&
		= \sum_{n=1}^{\infty}\frac{1-\e^{-2\pi^2rn^2}}{2\pi^2n^2}\left(1- \e^{-\pi^2(t-r)n^2}\right)^2
		\lesssim \sum_{n=1}^{\infty}\frac{1\wedge ((t-r)^2n^4)}{n^2}
		\lesssim \sqrt{|t-r|},
	\end{align*}
	where the last inequality holds by Lemma \ref{inequality1}. The proof is complete. 
\end{proof}

Next, we present a result about the size of certain heat integrals of powers of
the ``parabolic distance'' 
$|t-s|^{1/2}+|x-y|$ between space-time points $(t\,,x)$ and $(s\,,y)$.

\begin{lemma}\label{lem:parabolic}
	For all $q\in(0\,,1]$ and $T>0$ fixed,
	\[
		\int_0^t\d s\int_\T\d y\  [G_s(x\,,y)]^2 \left[ (t-s)^{q/2} 
		+ |y-x|^q \right]
		\lesssim t^{(1+q)/2},
	\]
	uniformly for every $t\in(0\,,T]$ and $x\in\T$.
\end{lemma}

\begin{proof}
	An inspection of \eqref{G} tells us that the integral of the lemma
	is independent of $x\in\T$. Therefore, we consider the case that $x=0$.
	Since $\int_\T[G_{t-s}(0\,,y)]^2\,\d y = G_{2(t-s)}(0\,,0)\lesssim (t-s)^{-1/2}$
	uniformly for all  
	$t\in(0\,,T]$
	[see \eqref{G:le}], we change variables in order to deduce the inequality
	\[
		\int_0^t\d s\int_\T\d y\  [G_{t-s}(0\,,y)]^2  s^{q/2}\lesssim t^{(1+q)/2},
	\]
	uniformly for $t\in(0\,,T]$. It remains to establish the same upper bound when
	the integrand $s^{q/2}$ is replaced by $|y|^q$.
	With this aim in mind, note that $|w+2n| \ge |n|$
	uniformly for every $n\in\Z\setminus\{0\}$ and $-1\le w\le 1$.
	Therefore, we may appeal to \eqref{G:phi} to deduce
	the heat-kernel estimate,
	\[
		\left| G_r(a\,,b) - \phi_r(b-a) \right| \le 
		\sum_{n\in\Z\setminus\{0\}}\frac{\e^{-n^2/(4r)}}{\sqrt{4\pi r}}
		\le \int_0^\infty\frac{\exp\{-w^2/(4r)\}}{\sqrt{\pi r}}\,\d w\le2,
	\]
	valid for all $a,b\in\T$ and $r>0$.
	Consequently, we may write
	\[
		\int_0^t\d s\int_\T\d y\  [G_{t-s}(0\,,y)]^2 |y|^q
		\lesssim
		\int_0^t\frac{\d s}{s}\int_{-\infty}^\infty\d y\  [\phi_1(y/\sqrt s)]^2 |y|^q + t,
	\]
	and change variables in the integral  in order  to conclude the proof.
\end{proof}

It is well known that the variance of every coordinate of
$H(t\,,x)$ is of order $\sqrt t$ when $t\ll1$. The next
result shows that the said variance is in fact nearly $\sqrt{t/\pi}$ when $t\ll1$.

\begin{lemma}\label{lem:Var(H)}
	$\sqrt{t/\pi}\le \Var H_1(t\,,x) \le \sqrt{t/\pi}+ (t/2)$
	for all $t\ge0$ and $x\in\T$.
\end{lemma}

\begin{proof}
	Thanks to stationarity and symmetry, we may and will assume that 
	$x=0$. Now, the semigroup property and the symmetry of the heat kernel together yield
	$\Var H_1(t\,,0) = \int_0^t\d s\int_\T\d y\
	[G_s(0\,,y)]^2 = \int_0^t G_{2s}(0\,,0)\,\d s$, which can be simplified as
	\[
		\int_0^t \frac{\d s}{\sqrt{4\pi s}}+
		2\sum_{n=1}^\infty\int_0^t 
		\e^{-n^2/s}\, \frac{\d s}{\sqrt{4\pi s}}
		= \sqrt{\frac{t}{\pi}} + 
		2\sum_{n=1}^\infty\int_0^t 
		\e^{-n^2/s}\, \frac{\d s}{\sqrt{4\pi s}}.
	\]
	This implies the result since the preceding infinite sum is on one hand $\ge0$
	and on the other hand bounded above by 
	$\int_0^\infty\d w
	\int_0^t  \exp(-w^2/s)\, \d s/\sqrt{4\pi s}=t/4.$
	This  proves the lemma.
\end{proof}

Our next result presents a way to quantify the  notion that 
$H(t\,,x)$ and $H(t\,,z)$ are close to being independent
when $t\ll1$ and $|x-z|\gg\sqrt t$.

\begin{lemma}
	Uniformly for all $t>0$ and $x,z\in\T$,
	\[
		\Cov[ H_1(t\,,x)\,,H_1(t\,,z)] \lesssim  \max( \sqrt t \,, t )\exp\left( - 
		\frac{|x-z|^2}{8t}
		\right).
	\]
	Consequently, there exists $c>1$ such that
	\begin{align*}
		&
			0< \inf_{t\in(0,1]}\inf_{|x-z|\ge c\sqrt t} t^{-1/2} \Var[H_1(t\,,x)-H_1(t\,,z)] \\
		&
			\hskip1in\le\sup_{t\in(0,1]}\sup_{x,z\in\T} t^{-1/2} \Var[H_1(t\,,x)-H_1(t\,,z)] 
		<\infty.
	\end{align*}
\end{lemma}

\begin{proof}
	Thanks to stationarity, we may (and will) assume
	without loss of generality, that $z=0$. Now, for every $t>0$ and $x\in\T$,
	we invoke the semigroup property of $G$ to find that
	$\mathscr{C} =\Cov[ H_1(t\,,x)\,,H_1(t\,,0)] $
	satisfies
	\[
		\mathscr{C} = \int_0^t G_{2s}(0\,,x)\,\d s
		\le  \sum_{n=0}^\infty\int_0^t\exp\left( - \frac{x^2 - 4n|x|+4n^2}{8s}\right)
		\frac{\d s}{\sqrt{2\pi s}};
	\]
	see \eqref{G}. Therefore,
	\begin{align*}
		\mathscr{C} &\le \int_0^t\exp\left( - \frac{x^2}{8s}\right)
			\frac{\d s}{\sqrt{2\pi s}}+\sum_{n=1}^\infty\int_0^t\exp\left( - \frac{x^2 + 4n(n-1)}{8s}\right)
			\frac{\d s}{\sqrt{2\pi s}}\\
		& \le \exp\left( - \frac{x^2}{8t}\right)\left[\sqrt{2t} + 
			\sum_{n=1}^\infty\int_0^t\exp\left( - \frac{n(n-1)}{2s}\right)
			\frac{\d s}{\sqrt{s}}\right].
	\end{align*}
	We may consider the sum before the integral. Since
	\[
		\sum_{n=2}^\infty\exp\left( - \frac{n(n-1)}{2s}\right) \le
		\sum_{n=2}^\infty\exp\left( - \frac{n^2}{4s}\right)\le
		\int_0^\infty\exp\left( - \frac{y^2}{4s}\right) \,\d y \propto \sqrt s,
	\]
	uniformly for all $s\in(0\,,1]$, the lemma follows.
\end{proof}

Armed with the preceding technical lemmas, we are ready to present one of the main results of this section.
The following proposition asserts that the Gaussian random field $H=\{H(t\,,x)\}_{(t,x)\in(0,\infty)\times\T}$ is 
strongly locally nondeterministic \cites{B73, MP87, X09}
and provides sharp bounds on conditional variances for $H$.
See \cite{LX}*{Proposition 5.1} and \cite{HL}*{Section 3.2} for similar results.

\begin{proposition}\label{pr:LND}
	For every $T > 0$,
	\begin{equation}\label{E:LND}\begin{aligned}
		&\Var(H_1(t\,,x)\mid H_1(t_1\,,x_1),\dots,H_1(t_m\,,x_m))\\
		&\hskip1in \asymp 
			\min\left\{ \sqrt{t}\, , \min_{1\le j\le m}\left( |t-t_j|^{1/2} + |x-x_j| \right)\right\}, 
	\end{aligned}\end{equation}
	uniformly for all $m\in\N$ and $(t\,,x), (t_1\,,x_1), \dots, (t_m\,,x_m) \in (0\,,T] \times \T$.
\end{proposition}

\begin{proof}
	Recall that for any centered Gaussian vector $(Z\,, Z_1\,, \dots\,, Z_m)$,
	\begin{equation}\label{fact:cvar}
		\Var(Z\mid Z_1\,, \dots\,, Z_m)
		= \inf_{a_1,\ldots,a_m\in\R}
		\left\| Z - \sum_{j=1}^m a_j Z_j \right\|_2^2\le \min_{0 \le j \le m}\left\| Z-Z_j \right\|_2^2,
	\end{equation}
	where $Z_0=0$. This
	implies that $\Var(Z\mid Z_1\,, \dots\,, Z_m)$ is bounded above by
	\[
		\min\left\{ \Var H_1(t\,,x) ~,~\min_{1 \le j \le m}
		\E\left(\left|H_1(t\,,x) - H_1(t_j\,,x_j)\right|^2\right) \right\}.
	\]
	Thus, Lemmas \ref{lem:G-G:x}, \ref{lem:G-G:t}, and \ref{lem:Var(H)}
	together yield the upper bound in \eqref{E:LND}.
	
	In order to prove the lower bound in \eqref{E:LND}, let 
	us choose and fix $m \in \N$,
	$(t\,,x), (t_1\,,x_1),\dots,(t_m\,,x_m) \in (0\,,T] \times \T$ 
	as well as numbers $a_1, \dots, a_m \in \R$, and define
	\begin{align}\label{E:LND:r}
		r = (\sqrt{T}\vee1)^{-1} \min\left\{ \sqrt{t}\,, 
		\min_{1\le j\le m}\left( |t-t_j|^{1/2}\vee |x-x_j|\right)\right\}.
	\end{align}
	Note that $r \le \sqrt{t/T} \le 1$.
	We may and will assume that $r>0$, for, otherwise there is nothing to prove.
	Thanks to \eqref{E:T_mild},
	\begin{align*}
		\mathscr{E}&:=\E\left( \left| H_1(t\,,x) - \sum_{j=1}^m a_j H_1(t_j\,,x_j)
			\right|^2 \right)\\
		&= \int_{-\infty}^\infty \d s  \int_\T \d y \left| 
			G_{t-s}(x\,, y)\1_{[0, t]}(s) - \sum_{j=1}^m a_j G_{t_j-s}(x_j\,, y) \1_{[0, t_j]}(s)
			\right|^2\\
		& = \sum_{j,k=0}^m \alpha_j \alpha_k \int_0^{t_j \wedge t_k} \d s  \int_\T \d y\ 
			G_{t_j-s}(x_j\,, y) G_{t_k-s}(x_k\,, y),
	\end{align*}
	where $t_0 = t$, $x_0 = x$, $\alpha_0 = 1$, and $\alpha_j = -a_j$ for $j = 1, \dots, m$.
	Thanks to the semigroup property of the heat kernel,
	\begin{align*}
		\int_\T G_{t_j-s}(x_j\,, y) 
            G_{t_k-s}(x_k\,, y) \d y
			&= G_{t_j+t_k-2s}(x_j\,,x_k)\\
		&= \sum_{n=-\infty}^\infty 
            \phi_{t_j+t_k-2s}  (x_j-x_k-2n).
	\end{align*}
	According to the Poisson summation formula 
	(Lemma \ref{lem:poisson}),
	\[
		\sum_{n=-\infty}^\infty \phi_{t_j+t_k-2s}(x_j-x_k-2n)
		=\frac12\sum_{n=-\infty}^\infty \e^{\pi i n (x_j-x_k) - \pi^2 n^2 (t_j+t_k-2s)}.
	\]
	Consequently,
	\begin{align}\label{E:H_V}
		 \mathscr{E}
			& = \frac12 \sum_{n=-\infty}^\infty
				\sum_{j,k=0}^m \alpha_j\alpha_k \e^{\pi in(x_j - x_k)}
				\int_0^{t_j\wedge t_k} \e^{-\pi^2n^2(t_j+t_k-2s)} \d s\\\notag
		&= \frac12 \sum_{n=-\infty}^\infty \int_{-\infty}^\infty
			\left|\sum_{j=0}^m \alpha_j\e^{\pi inx_j} \e^{-\pi^2n^2(t_j-s)}
			\1_{[0, t_j]}(s) \right|^2 \d s\\\notag
		&= (4\pi)^{-1}\sum_{n=-\infty}^\infty
			\int_{-\infty}^\infty \left| \sum_{j=0}^m \alpha_j
			\e^{\pi inx_j} \frac{\e^{-i\tau t_j} - \e^{-\pi^2n^2t_j}}{%
			\pi^2n^2-i\tau} \right|^2 \d \tau
			\quad\text{(Plancherel)}\\\notag
		&= (4\pi)^{-1} \sum_{n=-\infty}^\infty \int_{-\infty}^\infty
			\left|\left(\e^{-i\tau t}-\e^{-\pi^2n^2t}\right) - \right.\\\notag
		&\hskip1.2in \left.
			-\sum_{j=1}^m a_j \e^{\pi in(x_j-x)}\left(\e^{-i\tau t_j}-
			\e^{-\pi^2n^2t_j}\right) \right|^2 \dfrac{\d \tau}{\pi^4n^4+\tau^2}.
	\end{align}
	In order to derive a lower bound for \eqref{E:H_V}, we employ Fourier transforms for functions on $\T$.
	Recall that we denote the Fourier transform on $\T$ by $\F$ in order to distinguish it from the Fourier 
	transform on $\R$; see \eqref{E:R_FT} and \eqref{E:T_FT}.
	
	Let $\psi: \R \to \R$ be a smooth, symmetric test function supported in $[-\frac12\,,\frac12]$ with $\psi(0) = 1$.
	Let us write $\psi_r$ for its rescaled version
	\begin{align}\label{psi:r}
		\psi_r(y) = \psi(r^{-1}y) \quad \text{for all $y \in \R$.}
	\end{align}
	Display \eqref{psi:r} defines a function $\psi_r$ on $\R$.
	We now define a function $\Psi_r:\T\to\R$ on $\T$ by restricting it as follows:
	\begin{align}\label{psipsi}
		\Psi_r(x) = \psi_r(x) \quad \text{for all $x \in [-1\,,1]$.}
	\end{align}
	Recall that $\F$ denotes the Fourier transform on the torus,
	and consider the following expression:
	\begin{align}\label{E:LND:I}
		I& 
			= \sum_{n=-\infty}^\infty\F\Psi_r(n)\int_{-\infty}^\infty\d\tau\ \times\\\notag
		&\times
			\left[\left(\e^{-i\tau t}-\e^{-\pi^2n^2t}\right) - 
			\sum_{j=1}^m a_j \e^{\pi in(x_j-x)}\left(\e^{-i\tau t_j}-
			\e^{-\pi^2n^2t_j}\right) \right]\e^{i\tau t}
			\hat{\psi}_{r^2}(\tau)\\\notag
		&= \sum_{n=-\infty}^\infty \hat{\psi}_r(\pi n)\int_{-\infty}^\infty \d \tau\ \times
			\\\notag
		&\ \times
			\left[\left(\e^{-i\tau t}-\e^{-\pi^2n^2t}\right) - 
			\sum_{j=1}^m a_j \e^{\pi in(x_j-x)}\left(\e^{-i\tau t_j}-
			\e^{-\pi^2n^2t_j}\right) \right] \e^{i\tau t}\hat{\psi}_{r^2}(\tau);
	\end{align}
	see \eqref{E:T_FT} and \eqref{psipsi} for the second equality
	which holds since $0 < r \le 1$.
	Because $\hat\psi_r(y) = r\hat\psi(ry)$ for all $y \in \R$ and $r>0$,
	and since $\hat\psi$ is a function of rapid decrease, the sum
	in \eqref{E:LND:I} converges absolutely.
	In particular, we may appeal to the inverse Fourier transform on
	$\R$ to rewrite $I$ as follows:
	\begin{align*}
		I = 2\pi \sum_{n=-\infty}^\infty \F\Psi_r(n) &\Big[
			\left(\psi_{r^2}(0) - \psi_{r^2}(t) \e^{-\pi^2n^2 t}\right)\\
		&- \sum_{j=1}^m a_j \e^{\pi i n (x_j-x)} \left(\psi_{r^2}(t-t_j)
			- \psi_{r^2}(t) \e^{-\pi^2 n^2 t_j}\right) \Big].
	\end{align*}
	Next, we make two more observations about the quantity $I$.
	
	Firstly, observe that $t/r^2 \ge 1$ by the definition \eqref{E:LND:r}
	of $r$. This together with the fact that $\psi$ is supported in
	$[-\tfrac12\,,\tfrac12]$ implies that $\psi_{r^2}(t) = 0$, and hence
	\begin{align*}
		I &= 2\pi \sum_{n=-\infty}^\infty \F\Psi_r(n)
			\left[ 1 - \sum_{j=1}^m a_j \e^{\pi in(x_j-x)} \psi_{r^2}(t-t_j) \right]\\
		&= 4\pi \left[ \Psi_r(0) - \sum_{j=1}^m a_j \Psi_r(x_j-x) \psi_{r^2}(t-t_j) \right],
	\end{align*}
	where we have used \eqref{E:T_FT} to obtain the last equality.
	Furthermore, according to the definition of $r$, $|t-t_j| \ge r^2$
	or $|x-x_j| \ge r$ for every $j \in \{1,\dots, m\}$. This implies that
	$\Psi_r(x_j-x) \psi_{r^2}(t-t_j) = 0$, and hence,
	\begin{align}\label{E:LND:I1}
		I = 4\pi \Psi_r(0) = 4\pi.
	\end{align}
	
	Our second observation about the quantity $I$ is
	the following, which is a consequence of the preceding 
	and the Cauchy--Schwarz inequality:
	\begin{align*}
		|I|^2 &
			\le 4\pi \E\left(\left|H_1(t\,,x)-\sum_{j=1}^m a_j H_1(t_j\,,x_j)
			\right|^2\right) \times\\
		&\hskip1in\times
			\sum_{n=-\infty}^\infty \int_{-\infty}^\infty
			|\hat\psi_r(\pi n)|^2 |\hat\psi_{r^2}(\tau)|^2
			(\pi^4 n^4 + \tau^2)\, \d \tau.
	\end{align*}
	Since $\hat\psi_s(a) = s\hat\psi(sa)$ for all $s>0$ and $a\in\R$,
	and $\hat\psi$ is a function of rapid decrease, we have
	\begin{align*}
		&\sum_{n\in\Z\setminus\{0\}}\int_{-\infty}^\infty
			|\hat\psi_r(\pi n)|^2 |\hat\psi_{r^2}(\tau)|^2 \pi^4 n^4 \, \d \tau\\
		&\qquad \lesssim \sum_{n=1}^\infty\frac{r^2n^4}{1+r^6n^6}
			\int_{-\infty}^\infty \frac{r^4\d \tau}{1+r^4\tau^2}
			\lesssim r
			\int_{-\infty}^\infty \frac{\d \tau}{1+r^4\tau^2}
			= r^{-1} \int_{-\infty}^\infty \frac{\d \tau}{1+\tau^2},
	\end{align*}
	where in the second inequality we have used the following estimate
	\begin{align*}
	\sum_{n=1}^\infty\frac{n^4}{1+r^6n^6}&= \sum\limits_{n\leq 1/r}\frac{n^4}{1+r^6n^6} + \sum\limits_{n> 1/r}\frac{n^4}{1+r^6n^6} 
	\leq r^{-5} + \int_{1/r}^\infty \frac{y^4}{1+r^6y^6}\d y \asymp r^{-5}. 
	\end{align*}
	Similar arguments yield
	\begin{align*}
		&\sum_{n=-\infty}^\infty|\hat\psi_r(\pi n)|^2 \int_{-\infty}^\infty
			|\hat\psi_{r^2}(\tau)|^2 \tau^2\, \d \tau
			=2 \int_\T \d z \, |\Psi_r(z)|^2 \int_{-\infty}^\infty
			|\hat\psi_{r^2}(\tau)|^2 \tau^2\, \d \tau\\
		&\lesssim \int_{-\infty}^\infty \d y \, |\psi_r(y)|^2
			\int_{-\infty}^\infty \frac{r^4\tau^2 \, \d \tau}{1+r^8\tau^4}
			= r^{-1} \int_{-\infty}^\infty \d y \, |\psi(y)|^2
			\int_{-\infty}^\infty \frac{\tau^2 \, \d \tau}{1+\tau^4},
	\end{align*}
	where the implied constants in the preceding two display
	depend only on $\psi$ and do not depend on
	$m, t, t_1, \dots, t_m, x, x_1, \dots, x_m, a_1, \dots, a_m, r$.
	It follows from the preceding calculations that
	\begin{align}\label{E:LND:I2}
		|I|^2 \lesssim \E\left(\left|H_1(t\,,x)
		- \sum_{j=1}^m a_j H_1(t_j\,,x_j)\right|^2\right) \times r^{-1},
	\end{align}
	where the implied constant depends only on $\psi$.
	Finally, we combine \eqref{E:LND:I1} and \eqref{E:LND:I2}, and use \eqref{fact:cvar} to find that 
	\begin{align*}
		\Var(H_1(t\,,x)&\mid H_1(t_1\,,x_1),\dots,
            H_1(t_m\,,x_m))\\
		  &= \inf_{a_1, \dots, a_m \in \R}\E\left(
			\left|H_1(t\,,x)-\sum_{j=1}^m a_j H_1(t_j\,,x_j)\right|^2\right)
		\gtrsim r,
	\end{align*}
	with an implied constant that depends only on the function $\psi$ -- which was
	simply an artifact of the proof and does not have any bearing on the process $H$.
	In particular, the implied constant above does not depend on any of the parameters $m,t,t_1,\dots,t_m,x,x_1,\ldots,x_m,r$.
	Recall the definition of $r$ in \eqref{E:LND:r} and use the elementary inequality $a\vee b \ge (a+b)/2$ for all $a, b \ge 0$ to finish the proof.
\end{proof}

\section{Linearization}\label{s:linearization}

We will need to consider more general initial data than constants: Let us say that
the initial data $h:\T\to\R^p$ is random but independent of $\xi$.\footnote{We will
use the symbol $h$ consistently instead of $u_0$, in this section, in order to remind
that the initial data can be random though independent of $\xi$.}
We will tacitly 
assume that $h$ has regularity as well; see for example the upper bound of 
Lemma \ref{lem:mom:u} where we will need $\mathcal{S}_k(h)<\infty$ unless
the lemma is vacuous. The same happens in Lemmas \ref{lem:u-u:x}, \ref{lem:u-u:t},
\ref{lem:u:mod} and \ref{lem:I-sX}.
In any case, there is no problem with existence, regularity, etc., and the solution
can be written as a system of mild integral equations, 
\begin{equation}\begin{split}\label{mild}
	u_i(t\,,x) &
		= (\mathcal{G}_th_i)(x) +
		\int_{(0,t)\times\T} G_{t-s}(x\,,y) b_i(u(s\,,y)) \, \d s \, \d y\\
	&\quad + \sum_{j=1}^p\int_{(0,t)\times\T} G_{t-s}(x\,,y)
		\sigma_{i,j}(u(s\,,y))\, \xi_j(\d s\,\d y),
\end{split}\end{equation}
for $i = 1, \ldots, p$, where $\{\mathcal{G}_t\}_{t\ge0}$ denotes the  semigroup associated to $G$ 
[see \eqref{G}]. That is,
\[	(\mathcal{G}_0\psi)(x)=\psi(x)
	\quad\text{and}\quad
	(\mathcal{G}_t\psi)(x) = \int_\T G_t(x\,,y)\psi(y)\,\d y,
\]
for all $t>0$ and $x\in\T$
for all measurable scalar functions $\psi:\T\to\R$ for which the above integral converges absolutely.

The primary goal of this section is to prove that when $t$ is very small, the conditional
law of $\{u(t\,,x)\}_{x\in\T}$ is
approximately that of a certain $p$-dimensional Gaussian random field, given the initial data $h$;
see Lemma \ref{lem:I-sX} and its seemingly stronger, but in fact essentially equivalent, 
consequence \eqref{E:III} for a precise formulation. 
When $p=1$, such a result appears in various forms in da Prato and Zabczyk \cite{DZ}
and Walsh \cite{W}, and many subsequent papers. Here, we prove improvements
of these results which are valid for general $p\ge 1$ and will appear in essentially-optimal form,
have essentially-optimal error-rate estimates, and possess optimal assumptions on the diffusion
coefficient $\sigma$.

For the following two lemmas, it might help to recall that
$\mathcal{S}_k(\cdot)$, $\mathcal{H}_k(\cdot)$, and $ \mathcal{M}(\cdot)$ are defined in \eqref{S:H}
and \eqref{Lip}, respectively.

\begin{lemma}\label{lem:mom:u}
	Choose and fix some $T>0$. Then, the following is valid
	\[
		\sup_{x\in\T}\| u(t\,,x)\|_k \lesssim \mathcal{S}_k(h)
		+ \mathcal{M}(b) \sqrt{t} \left( t^{1/4}\vee t^{1/2} \right)
		+ \mathcal{M}(\sigma)\sqrt{k}\left(t^{1/4}\vee t^{1/2}\right),
	\]
	where the implied constant does not depend on $\sigma$, $h$,
	$t\in[0\,,T]$, or $k\in[2\,,\infty)$.
\end{lemma}

\begin{proof} 
	The lemma has content 
	iff $\mathcal{S}_k(h)\vee\mathcal{M}(b)\vee\mathcal{M}(\sigma)<\infty$, 
	a condition which we assume.
	Because $\|(\mathcal{G}h)(x)\|_k\le\mathcal{S}_k(h)$,
	the Minkowski inequality, Cauchy--Schwarz inequality, and a
	suitable application of the Burkholder--Davis--Gundy inequality 
	(Lemma \ref{lem:BDG} below) together show that
	$\|u(t\,,x)\|_k$ is bounded from above by
	\begin{align*}
		&\mathcal{S}_k(h) + \int_{(0,t)\times\T}
			\|G_{t-s}(x\,,y) b(u(s\,,y))\|_k \, \d s \, \d y \\
			&\hskip1.8in + 
			\left\| \int_{(0,t)\times\T} G_{t-s}(x\,,y)
			\sigma(u(s\,,y))\, \xi(\d s\,\d y)\right\|_k\\
		&\le \mathcal{S}_k(h) + \mathcal{M}(b)
			\sqrt{2t} \left[ \int_0^t \d s \int_\T \d y\,
			[G_{t-s}(x\,,y)]^2 \right]^{1/2} \\
			&\hskip1.8in + \mathcal{M}(\sigma)
			\sqrt{4kp}\left[\int_0^t\d s\int_{\T}\d y\
			[G_{t-s}(x\,,y)]^2 \right]^{1/2}\\
		&= \mathcal{S}_k(h) + \mathcal{M}(b)\sqrt{2t}
			\left[\int_0^t G_{2s}(0\,,0)\,\d s \right]^{1/2} +
			\mathcal{M}(\sigma)\sqrt{4kp}\left[\int_0^t
			G_{2s}(0\,,0)\,\d s \right]^{1/2},
	\end{align*}
	thanks to the semigroup property of the heat kernel. In particular, \eqref{G:le} yields 
	the lemma.
\end{proof}

\begin{lemma}\label{lem:u-u:x}
	For every $T>0$,
	\[
		\| u(t\,,x) - u(t\,,z) \|_k \lesssim \mathcal{H}_k(h)\sqrt{|x-z|}
		+ \left(\mathcal{M}(b)+\mathcal{M}(\sigma)\sqrt{k}\right)\left[t^{1/4}\wedge \sqrt{|x-z|}\right],
	\]
	uniformly in $\sigma$, $h$,
	$t\in[0\,,T]$, $x,z\in\T$, and $k\in[2\,,\infty)$.
\end{lemma}

\begin{proof}
	Recall \eqref{S:H}. 
	Without loss of generality, we assume
	that $\mathcal{H}_k(h)\vee\mathcal{M}(b)\vee\mathcal{M}(\sigma)<\infty$.
	Minkowski's inequality shows that, for every $t>0$, $x,z\in\T$,
	and $k\in[2\,,\infty)$,
	\[
		\| (\mathcal{G}_th)(x) - (\mathcal{G}_th)(z)\|_k \le
		\int_\T G_t(0\,,y) \|  h(y+x) - h(y+z) \|_k\,\d y\le
		\mathcal{H}_k(h)\sqrt{|x-z|}.
	\]
	Therefore, a suitable appeal to the BDG inequality 
	[Lemma \ref{lem:BDG}] shows that 
	\begin{align*}
		&\| u(t\,,x) - u(t\,,z)\|_k\\
		&\le\mathcal{H}_k(h) \sqrt{|x-z|}
			+ \int_0^t \d s \int_\T \d y\, |G_{t-s}(x\,,y) - G_{t-s}(z\,,y)| \|b(u(s\,,y))\|_k\\
		&\qquad+ \sqrt{4kp} \left[ 
            \int_0^t\d s\int_\T\d y 
			\left[ G_{t-s}(x\,,y) - G_{t-s}(z\,,y)\right]^2
			\|\sigma(u(s\,,y))\|_k^2\right]^{1/2}\\
		&\le \mathcal{H}_k(h) \sqrt{|x-z|}
			+\left[ \mathcal{M}(b) \sqrt{2t} + \mathcal{M}(\sigma)
			\sqrt{4kp}\right] \\
		&\qquad\times\left[ \int_0^t\d s\int_\T\d y
			\left[ G_{t-s}(x\,,y) - G_{t-s}(z\,,y)\right]^2
			\right]^{1/2}.
	\end{align*}
	Lemma \ref{lem:G-G:x} completes the proof.
\end{proof}

\begin{lemma}\label{lem:u-u:t}
	For every $T>0$,
	\[
		\sup_{x\in\T}\| u(t\,,x) - u(r\,,x) \|_k \lesssim \left(\mathcal{H}_k(h)
		+ \mathcal{M}(b) + \mathcal{M}(\sigma)\sqrt{k}\right) |t-r|^{1/4},
	\]
	uniformly in $t, r\in[0\,,T]$, $\sigma$, $h$, and $k\in[2\,,\infty)$.
\end{lemma}

\begin{proof}
	Recall \eqref{S:H} and assume without loss of generality that
	$\mathcal{H}_k(h)\vee\mathcal{M}(b)\vee\mathcal{M}(\sigma)<\infty$.
	Choose and fix $T>0$. Let $t>r$ in $[0\,,T]$, and write for every 
	$x\in\T$, 
	\begin{equation}\label{u-u:I1:I2:I3}
		\| u(t\,,x) - u(r\,,x)\|_k  = I_1 +I_2+I_3, 	
	\end{equation}
	where
	\begin{align*}
		I_1 &= (\mathcal{G}_th)(x)-(\mathcal{G}_rh)(x),\\
		I_2 &= \int_{(0,r)\times\T} \left[ G_{t-s}(x\,,y) - G_{r-s}(x\,,y) \right] b(u(s\,,y)) \, \d s \, \d y \\
		& \hskip1in + \int_{(0,r)\times\T} \left[ G_{t-s}(x\,,y)-G_{r-s}(x\,,y)\right] 
			\sigma(u(s\,,y))\, \xi(\d s\,\d y),\\
		I_3 &= \int_{(r,t)\times\T} G_{t-s}(x\,,y) b(u(s\,,y)) \, \d s \, \d y
			+ \int_{(r,t)\times\T} G_{t-s}(x\,,y)
			\sigma(u(s\,,y))\, \xi(\d s\,\d y).
	\end{align*}
	Recall that $\mathcal{G}$ can be identified with the semigroup
	for a Brownian motion $\{b(s)\}_{s\ge0}$ on the torus. Because
	time increments of $b$ have the same laws as the same time increments 
	but for a Brownian motion on $\R$, \eqref{S:H} implies that for $k\in[2\,,\infty)$,
	\[
	\begin{split}
		\|I_1\|_k &\le \E \left\| h(x+b(t)) - h(x+b(r))\right\|_k \\
		&\le \mathcal{H}_k(h)
		\E\sqrt{|b(t)-b(r)|}\lesssim
		\mathcal{H}_k(h) |t-r|^{1/4},
		\end{split}
	\]
	where the implied constant is universal. This is the desired estimate for $I_1$. 
	
	In order to bound $I_2$, we appeal to the BDG inequality [Lemma \ref{lem:BDG}]
	and Lemma \ref{lem:G-G:t} to obtain the following:
	\begin{align*}
		\|I_2\|_k &
			\le \int_0^r \d s \int_\T \d y\, |G_{t-s}(x\,,y)-G_{r-s}(x\,,y)| \|b(u(s\,,y))\|_k\\
		& \quad+ \left[ 4kp \int_0^r\d s\int_\T\d y
			\left| G_{t-s}(x\,,y)-G_{r-s}(x\,,y)\right|^2 \|\sigma(u(s\,,y))\|_k^2\right]^{1/2}\\
		&\le\left[ \mathcal{M}(b) \sqrt{2t} + 
			\mathcal{M}(\sigma) \sqrt{4kp} \right] \left[\int_0^r\d s\int_\T\d y
			\left| G_{t-s}(x\,,y)-G_{r-s}(x\,,y)\right|^2\right]^{1/2}\\
		&\lesssim \left[\mathcal{M}(b) + 
			\mathcal{M}(\sigma)\sqrt{k}\right](t-r)^{1/4},
	\end{align*}
	where the implied constant does not depend on $(t\,,r\,,x\,,k\,,\sigma\,,h)$.
	
	We proceed in a similar fashion to bound $I_3$ as follows:
	\begin{align*}
		\|I_3\|_k &\le \int_r^t \d s \int_\T \d y \|G_{t-s}(x\,,y) b(u(s\,,y))\|_k\\
			&\qquad \qquad + \left[4kp \int_r^t\d s\int_\T\d y\
			[G_{t-s}(x\,,y)]^2\|\sigma(u(s\,,y))\|_k^2\right]^{1/2}\\
		&\le \left[ \mathcal{M}(b) \sqrt{2t} +
			\mathcal{M}(\sigma) \sqrt{4kp} \right] \left[\int_r^t\d s\int_\T\d y\
			[G_{t-s}(x\,,y)]^2 \right]^{1/2}\\
		&\lesssim \left[ \mathcal{M}(b) + \mathcal{M}(\sigma)\sqrt{k}
			\right]\left[\int_0^{t-r}G_{2s}(0\,,0)\,\d s \right]^{1/2}\\
			&\lesssim\left[ \mathcal{M}(b) + \mathcal{M}(\sigma)\sqrt{k} \right] (t-r)^{1/4},
	\end{align*}
	where the last two inequalities follow from the semigroup property 
	of the heat kernel and \eqref{G:le}, and the implied constants do 
	not depend on the parameters $(\sigma\,,h\,,t\,,r\,,x\,,k)$. Combine the preceding norm bounds
	for $I_1,I_2,I_3$ and use \eqref{u-u:I1:I2:I3} to complete the proof.
\end{proof}

Lemmas \ref{lem:u-u:x} and \ref{lem:u-u:t}, and
an appeal to Dudley's metric-entropy theorem \cite{Dudley}
together yield the following. The proof is omitted as it is nowadays standard.

\begin{lemma}\label{lem:u:mod}
	Choose and fix some $T>0$.
	If there exists $L>0$ such that $\mathcal{H}_k(h)\le L\sqrt{k}$ for all $k\in[2\,,\infty)$, 
	then there exists $c=c(L\,,\mathcal{M}(b)\,,\mathcal{M}(\sigma)\,,T)>0$
	such that
	\[
		\E\left[\exp\left( c 
		\sup_{\substack{(t,x),(s,z)\in[0,T]\times\T\\ (t,x)\neq (s,z)}}
		\left| \frac{\|u(t\,,x) - u(s\,,z)\|}{%
		\rho( (t\,,x)\,,(s\,,z) ) \sqrt{\log_+  
		(1/\rho( (t\,,x)\,,(s\,,z)) ) |}}
		\right|^2 \right)\right]<\infty,
	\]
	where $\rho$ was  defined in \eqref{def:rho}.
\end{lemma}

For every $t>0$ and $x\in\T$, define
\begin{equation}\label{I(t,x)}\begin{split}
	I(t\,,x) &
		= \int_{(0,t)\times\T} G_{t-s}(x\,,y) b(u(s\,,y)) \, \d s\, \d y\\
	& \hskip1in +  \int_{(0,t)\times\T} G_{t-s}(x\,,y)\sigma(u(s\,,y))\, \xi(\d s\,\d y).
\end{split}\end{equation}

\begin{lemma}\label{lem:I-sX}
	For every $T>0$,
	\[
		\| I(t\,,x) - \sigma(h(x))H(t\,,x) \|_k
	\lesssim [\mathcal{H}_k(h)\vee\mathcal{M}(b)\vee\mathcal{M}(\sigma)] k \sqrt t,
	\]
	where the implied constant is independent of $t\in(0\,,T]$, $x\in\T$, $k\in[2\,,\infty)$ 
	and $h$, and depends on $\sigma$ only through its Lipschitz constant $\lip(\sigma)$.
\end{lemma}

\begin{proof}
	Recall \eqref{S:H} and assume without loss of generality that
	$\mathcal{H}_k(h)\vee\mathcal{M}(b)\vee\mathcal{M}(\sigma)<\infty$.
	For every $t>0$ and $x\in\T$, write
	\begin{align*}
		&
			I(t\,,x) - \sigma(h(x))H(t\,,x)=
			\int_{(0,t)\times\T} G_{t-s}(x\,,y) b(u(s\,,y)) \, \d s\, \d y\\
		&\hskip1.5in
			+\int_{(0,t)\times\T} G_{t-s}(x\,,y)\left[\sigma(u(s\,,y)) - \sigma(h(x))\right]
		 	\xi(\d s\,\d y).
	\end{align*}
	We first apply the Minkowski inequality and \eqref{G:le} in order to see that
	\begin{align*}
		&\left\| \int_{(0,t)\times\T}
			G_{t-s}(x\,,y) b(u(s\,,y)) \, \d s\, \d y \right\|_k
			\le \int_0^t \d s \int_\T \d y \, \|G_{t-s}(x\,,y) b(u(s\,,y))\|_k \\
		&\qquad\le \mathcal{M}(b) \int_0^t \d s \int_\T  \d y G_{t-s}(x\,,y)
			\lesssim \mathcal{M}(b) \int_0^t \max(s^{-1/2}, 1)\,
			\d s \lesssim \mathcal{M}(b) \sqrt{t},
	\end{align*}
	where the implied constant depends only on $T$.
	
	Next, observe that if $A$ and $B$ are two random variables with values in $\R^p$, then
	\[
		\| \sigma(A) - \sigma(B) \|_k \le p^{1/k}\lip(\sigma)\|A-B\|_k.
	\]
	We apply this with $A=u(s\,,y)$ and $B=h(x)$,
	together with the Burkholder--Davis--Gundy inequality in the form of
	Lemma \ref{lem:BDG} below, in order to see that
	\begin{align*}
		&
			\left\| \int_{(0,t)\times\T} G_{t-s}(x\,,y)\left[\sigma(u(s\,,y)) - \sigma(h(x))\right]
		 	\xi(\d s\,\d y) \right\|_k\\
		&\le \left[4kp\int_0^t\d s\int_\T\d y\ 
			[G_{t-s}(x\,,y)]^2 \left\| \sigma(u(s\,,y)) - \sigma(h(x))\right\|_k^2\right]^{1/2}\\
		&\le \sqrt{4kp}\lip(\sigma)
			\left[\int_0^t\d s\int_\T\d y\  [G_{t-s}(x\,,y)]^2
			\left\| u(s\,,y) - h(x)\right\|_k^2 \right]^{1/2}.
	\end{align*}
	As regards the $L^k(\P)$-norm inside the integral, we have, thanks to 
	Lemma \ref{lem:u-u:t} and \eqref{S:H},
	\begin{align*}
		\left\| u(s\,,y) - h(x)\right\|_k
			&\le \left\| u(s\,,y) - h(y)\right\|_k
			+ \| h(y) - h(x) \|_k\\
		&\lesssim (\mathcal{H}_k(h) + \mathcal{M}(b)+\mathcal{M}(\sigma)\sqrt{k})\, s^{1/4}
			+\mathcal{H}_k(h) \sqrt{|y-x|},
	\end{align*}
	where the implied constant is independent of $s\in(0\,,T]$, $x,y\in\T$, $k\in[2\,,\infty)$
	and $h$, and depends on $\sigma$ only through $\lip(\sigma)$. We can combine the preceding
	two displays to find that
	\begin{align*}
		&
			\left\| \int_{(0,t)\times\T} G_{t-s}(x\,,y)\left[\sigma(u(s\,,y)) - \sigma(h(x))\right]
			 \xi(\d s\,\d y) \right\|_k^2\\
		&\lesssim k[\mathcal{H}_k(h)\vee\mathcal{M}(b)\vee\mathcal{M}(\sigma)]^2\left\{
			\int_0^t\d s\int_\T\d y\  [G_{t-s}(x\,,y)]^2 \left[ 
			k\sqrt{s}+ |y-x| \right]\right\}\\
		&\lesssim  [\mathcal{H}_k(h)\vee\mathcal{M}(b)\vee\mathcal{M}(\sigma)]^2 k^2 t,
	\end{align*}
	thanks to Lemma \ref{lem:parabolic}, where the parameter dependencies are
	as in the statement of the lemma. This concludes the proof of Lemma \ref{lem:I-sX}.
\end{proof}

\section{Sufficient conditions for uniform dimension}

The primary goal of this section is to present results that streamline the process of establishing 
uniform dimension theorems for space-time random fields.

\begin{lemma}\label{lem:unif:dim}
	Consider an $\R^p$-valued random field $\{ Z(s\,,y) \}_{(s,y) \in [S,T]\times\T}$, 
	where $T>S\ge 0$. Choose and fix some number $\alpha \in (0\,, 1]$. For every
	$n \in \N$, $\delta \in (0\,,1)$, $y \in\T$, $\nu \in \R^p$, and $r > 0$, define
	\begin{equation}\label{E:Z:N}\begin{split}
		&\mathcal{N}_n^\delta(s\,, B(\nu\,,r)) = \# \{ y \in F_n^\delta : Z(s\,,y) \in B(\nu\,,r) \},\\
		&\text{where} \quad F_n^\delta = \T \cap \{ j 2^{-n(1+\delta)/\alpha} : j \in \Z \}.
	\end{split}\end{equation}
	Recall that $B(\nu\,,r)$ is the closed ball defined in \eqref{E:B}.
	Suppose the following two conditions hold:
	\begin{compactenum}[(i)]
	\item For every $\varepsilon\in(0\,,1), 
		\P\{Z(s)\in C^{\alpha(1-\varepsilon)}(\T)\,\forall s\in[S\,,T]\}=1$;
	\item For every $ R > 0$, 
	$$\lim_{\delta\to0+}\limsup_{n\to\infty} n^{-1}
		\sup_{s \in[S,T]} \sup_{\nu \in B(0, R)}\log
		\mathcal{N}_n^\delta(s\,, B(\nu\,, 2^{-n})) =0\ \hbox{ a.s.}
		$$
	\end{compactenum}
	Then, $\P\{\dimh  Z(\{s\}\times F) = \alpha^{-1} \dimh  F \ 
	\forall\text{compact } F \subset\T,\, s \in [S\,,T]\}=1$.
\end{lemma}

\begin{proof}
	By Lemma \ref{lem:easy}, \emph{(i)} implies that
	$\dimh Z(\{s\}\times F)\le\alpha^{-1}\dimh F$ for all compact $F$ and $s\in[S\,,T]$,
	all valid off a single $\P$-null set.
	In the remainder of the proof, we aim to show that, off a single $\P$-null set,
	\begin{equation}\label{Eq:extra}
		\dimh  Z(\{s\}\times F) \ge \alpha^{-1}\dimh  F \quad \forall
		\text{compact  $F\subset \T,\ s \in [S\,,T]$.}
	\end{equation}
	Choose and fix an arbitrary $R>0$ and observe that 
	taking  $F= Z(s)^{-1}(A) =\{ y\in\T:\, Z(s\,,y)\in A\}$  in (\ref{Eq:extra})
	yields the following equivalent formulation: 
	Off a single $\P$-null set,
	\begin{equation}\label{E:dim:LB}
		\dimh Z(s)^{-1}(A)\le \alpha 
		\dimh A \quad \forall
		\text{compact  $A\subset B(0\,,R),\, s \in [S\,,T]$.}
	\end{equation}
	In order to establish \eqref{E:dim:LB}, we fix $\delta\in(0\,,1)$ 
	and use the condition of part \emph{(ii)} in order to choose a $\P$-null set $\Omega_0$ off which
	\begin{equation}\label{E:lem:N:as}
		\adjustlimits \sup_{s \in [S,T]}\max_{\nu\in B(0,R)}
		\mathcal{N}_n^\delta(s\,,B(\nu\,,2^{-n})) = 
		\mathcal{O}(2^{c_\delta n})\quad\text{as $n\to\infty$},
	\end{equation}
	where $c_\delta \to 0$ as $\delta \to 0$.
	Also, we can find a $\P$-null set $\Omega_1$ off which
	the condition of part \emph{(i)} holds for 
	\begin{align}\label{Omega1:eps}
		\varepsilon = \varepsilon(\delta) = 1- \frac{1+(\delta/2)}{1+\delta}.
	\end{align}
	Now that we have identified the null set $\Omega_2 := \Omega_0\cup\Omega_1$
	(which depends on $\delta$),
	we work pathwise, nonprobabilistically, from here on in order
	to conclude the proof of \eqref{E:dim:LB}, and hence the lemma.
	
	To this end, we deduce from the definition of Hausdorff dimension that
	for every Borel set $A\subset B(0\,,R)$ 
	and for all $\kappa>\dimh A$
	and $n\in\N$, we can find 
	$\nu_{1,n},\nu_{2,n},\ldots\in B(0\,,R)$ and $r_{1,n},r_{2,n},\ldots\in(0\,,2^{-n})$ such that
	$A\subset\cup_{i=1}^\infty \mathbb{B}(\nu_{i,n}\,,r_{i,n})$ -- see \eqref{E:B} --
	and 
	\begin{equation}\label{E:lem:HDim}
		\sup_{n\in\N}\sum_{i=1}^\infty  r_{i,n}^\kappa<\infty. 
	\end{equation}
	Since $\{\mathbb{B}(\nu_{i,n}\,,r_{i,n})\}_{i=1}^\infty$ is an open cover of
	$A$, $\{ Z(s)^{-1}(\mathbb{B}(\nu_{i,n}\,,r_{i,n}))\}_{i=1}^\infty$
	is an open cover of $Z(s)^{-1}(A)$ for every $s \in [S\,,T]$. 
	
	We observe that, off $\Omega_2$, there exists a random number 
	$C>0$ and a non-random number $C_0>0$ 
	such that for all sufficiently large $n\in\N$, every inverse image 
	of the form
	$Z(s)^{-1}(\mathbb{B}(\nu_{i,n}\,,r_{i,n}))$ lies
	in not more than $C r_{i,n}^{-c_\delta}$-many intervals of length $C_0 r_{i,n}^{(1+\delta)/\alpha}$.
	In order to see this, consider a large $n\in\N$, fix $i\in\N$, and 
	let $y \in Z(s)^{-1}(\mathbb{B}(\nu_{i, n}, r_{i, n}))$, where 
	$2^{-m-1} \le r_{i, n} \le 2^{-m}$ for some $m \ge n$. 
	Then, the point $y$ is contained in an interval of side length 
	$2^{-(m-1)(1+\delta)/\alpha} [\le C_0 r_{i, n}^{(1+\delta)/\alpha}\,]$ 
	centered at a point $y^* \in F_{m-1}^\delta$ that is nearest to $y$.
	We see that the center $y^*$ of this interval must belong to 
	$Z(s)^{-1}(\mathbb{B}(\nu_{i, n}, 2^{-m+1})) \cap F_{m-1}^\delta$
	uniformly for all large $m$. 
	This is because the choice of $\varepsilon$ in \eqref{Omega1:eps} together with the triangle 
	inequality implies that
	\begin{align*}
		&\|Z(s\,,y^*) - \nu_{i,n}\| \le \|Z(s\,,y) - \nu_{i,n}\| + \|Z(s\,,y^*) - Z(s\,, y)\|\\
		&\qquad\lesssim 2^{-m} +  (2^{-(m-1)(1+\delta)/\alpha})^{\alpha(1-\varepsilon)} 
			\le 2^{-m} +  2^{-(1+\delta/2)(m-1)}\le 2^{-m+1},
	\end{align*}
	uniformly for all sufficiently large $m\in\N$, and where
	the implied constant does not depend on any of the parameters that arise. According
	to \eqref{E:lem:N:as},
	\[
		\#[Z(s)^{-1}(\mathbb{B}(\nu_{i, n}, 2^{-m+1}))
		\cap F_{m-1}^\delta] \le C r_{i, n}^{-c_\delta}
		\quad \forall n \text{ large enough}.
	\]
	This shows that, when $n$ is sufficiently large, 
	$Z(s)^{-1}(\mathbb{B}(\nu_{i,n}\,,r_{i,n}))$ lies in not more than 
	$C r_{i,n}^{-c_\delta}$-many intervals of side length 
	$C_0 r_{i,n}^{(1+\delta)/\alpha}$, valid uniformly for all $s\in[S\,,T]$.
	
	Next, choose and fix $\theta>0$ to see that, off $\Omega_2$, 
	the $\theta$-dimensional Hausdorff measure of $Z(s)^{-1}(A)$
	is at most
	\[
		C\sum_{i=1}^\infty r_{i,n}^{\theta(1+\delta)/\alpha - c_\delta},
	\]
	regardless of 
	the value of $n$, the choice of the Borel set
	$A\subset B(0\,,R)$, or the value of $s\in[S\,,T]$. 
	Thanks to \eqref{E:lem:HDim}, the preceding sum converges 
	uniformly in $n$ provided that  
	$\theta(1+\delta)/\alpha - c_\delta >\kappa$.
	Since $\kappa > \dimh A$ is arbitrary, this proves that,
	off $\Omega_2$, 
	\[
		\dimh Z(s)^{-1}(A) \le \frac{\alpha(\dimh A + c_\delta)}{1+\delta},
	\]
	simultaneously for all compact  $A\subset B(0\,,R)$ and $s \in [S\,,T]$.
	Let $\delta\downarrow0$ along a rational
	sequence in order to deduce that \eqref{E:dim:LB} holds off a single $\P$-null set.
	Since $R>0$ is arbitrary, this completes the proof of Lemma \ref{lem:unif:dim}.
\end{proof}

The next lemma provides a sufficient condition for 
part \emph{(ii)} of the previous lemma to be applicable.

\begin{lemma}\label{lem:interpolation}
	Choose and fix some $T>S\ge 0$,
	and let $\{ Z(s\,,y) \}_{(s,y) \in [S,T]\times\T}$ be an $\R^p$-valued
	random field. Recall \eqref{def:rho} and \eqref{E:Z:N} with $\alpha=\frac12$,
	and suppose that:
	\begin{compactenum}[(i)]
	\item The following is finite for some $c,\, \gamma>0$:
		\[
			\mathcal{E}_0:=\E\left[ \exp\left(c\sup
			\left|\frac{\|Z(s\,,y)-Z(s',y')\|}{\rho((s\,,y)\,,(s',y'))\sqrt{\log_+
			(1/\rho((s\,,y)\,,(s',y')))}}\right|^\gamma\right)\right],
		\]
		where the sup is over all distinct points
		$(s\,,y),(s',y')\in[S\,,T]\times\T$;
	\item For every $\delta\in(0\,,1)$, there exist numbers $a,\kappa>0$ such that
		\[
			\P\left\{\mathcal{N}_n^\delta(s\,,B(\nu\,,2^{-n}))
			\ge 2^{an}\right\} \le 2^{-\kappa n^2},
		\]
		uniformly for all $n\in\N$, $s\in[S\,,T]$ and $\nu\in \R^p$.\smallskip
		\end{compactenum}
		Then, for every $\delta\in(0\,,1)$ and $R>0$, there exist 
		$K,L>0$ such that
		\[ 
			\P\left\{ \sup_{s \in [S,T]} \sup_{\nu \in B(0,R)}
			\mathcal{N}_n^\delta(s\,,B(\nu\,,2^{-n})) \ge K2^{an} \right\} \le L^n
			\e^{- n^2/L},
		\]
		uniformly for all $n \in \N$.
\end{lemma}

\begin{proof}
	Fix $\delta\in(0\,,1)$ and $R>0$.
	Define
	\[
		A(n) =\left\{ \frac{j}{2^{4n(1+\delta)}}: j\in\Z\right\}
		\quad \text{and} \quad
		D(n) = \left\{ \frac{\ell}{2^{2n}\sqrt{p}}: \ell\in\Z^p\right\}.
	\]
	First, note that there is a number $K>0$, depending only on 
	$\delta$, such that for every $n\in\N$ and every $y^*\in F_{n-1}^\delta$,
	\begin{align}\label{E:K}
		\#\left(B(y^*,2^{-2(n-1)(1+\delta)})\cap F_n^\delta\right) \le K.
	\end{align}
	We can observe that the following holds for all $n\in\N$:
	\begin{align}\notag
		&\P\left\{ \sup_{\substack{s\in [S,T],\\\nu \in B(0,R)}}
		\mathcal{N}_n^\delta(s\,,B(\nu\,,2^{-n})) \ge K2^{an},
		\sup_{s,s',y,y'} \|Z(s\,,y)-Z(s',y')\| \le 2^{-n-1} \right\}\\
		&\le \P\left\{ \adjustlimits\max_{s\in A(n) \cap [S,T]}\max_{\nu\in D(n)\cap B(0,R)}
		\mathcal{N}_{n-1}^\delta(s\,,B(\nu\,,2^{-n+1})) \ge 2^{an}\right\},
		\label{E:p:supN:maxN}
	\end{align}
	where $\sup_{s,s',y,y'}$ denotes, here and throughout the proof,
	the sup operator over all $(s\,,y),(s',y')\in [S\,,T]\times\T$
	that satisfy $|s-s'|\le 2^{-4n(1+\delta)}$ and $|y-y'|\le
	16\times 2^{-2n(1+\delta)}$. In order to see why \eqref{E:p:supN:maxN}
	is true, let us
	suppose that there exist $s \in [0\,,T]$, $\nu \in B(0\,,R)$
	and $y_1, \dots, y_m \in F_{n}^\delta$ 
	with $\#\{y_1, \dots, y_m\} = m \ge K2^{an}$ such that 
	$\|Z(s\,, y_i) - \nu\| \le 2^{-n}$ for all $1\le i \le m$, and that
	$\sup_{s,s',y,y'} \|Z(s\,,y)-Z(s',y')\| \le 2^{-n-1}.$
	Then, we can choose nearest neighbors $s^* \in A(n) \cap [S\,,T]$, 
	$\nu^* \in D(n) \cap B(0\,,R)$ and $y_1^*, \dots, y_m^* \in F_{n-1}^\delta$ 
	that respectively satisfy 
	$|s^*-s| \le 2^{-4n(1+\delta)}$, $\|\nu^*-\nu\| \le 2^{-2n}$ 
	and $|y_i^* - y_i| \le 2^{-2(n-1)(1+\delta)} [\le 16\times  2^{-2n(1+\delta)}]$, 
	in order to see that
	\begin{align*}
		\|Z(s^*,y_i^*) - \nu^*\| &\le \|Z(s^*,y_i^*) - Z(s\,,y_i)\| + \|Z(s\,,y_i) - \nu\| + \|\nu - \nu^*\| \\
		& \le 2^{-n-1} + 2^{-n} + 2^{-2n} \le 2^{-n+1}.
	\end{align*}
	By \eqref{E:K}, $\#\{y_1^*, \dots, y_m^*\} \ge m/K$, whence it follows that
	$\mathcal{N}_{n-1}^\delta(s^*,B(\nu^*,2^{-n+1})) $ is at least 
	$\#\{y_1^*, \dots, y_m^*\} 
	\ge m/K \ge 2^{an}.$
	This proves \eqref{E:p:supN:maxN}.
	
	Next, condition \emph{(ii)}
	together with a simple union bound yields the following, valid uniformly for
	all $n\in\N$:
	\begin{equation}\label{E:p:maxN}
		\P\left\{ 
        \max_{\substack{s\in A(n)\cap [S,T]\\\nu\in D(n)\cap B(0,R)}}
		\mathcal{N}_{n-1}^\delta(s\,,B(\nu\,,2^{-n+1})) \ge 2^{an} \right\}
		\le L_1^n\e^{-n^2/L_1},
	\end{equation}
	for a constant $L_1>0$ that depends only on $(p\,,S\,,T\,,\delta\,,R\,,\kappa)$.

	Now, we can put together \eqref{E:p:maxN} and \eqref{E:p:supN:maxN} 
	in order to conclude that
	\begin{equation}\label{E:p1p2:F(n)}
		\P\left\{ \adjustlimits\sup_{s\in [S,T]}\sup_{\nu \in B(0,R)}
		\mathcal{N}_n^\delta(s\,,B(\nu\,,2^{-n})) \ge K2^{an}\right\} \le
		L_1^n \e^{-n^2/L_1}+ \mathscr{P}_n,
	\end{equation}
	where $\mathscr{P}_n=\P\{\sup_{s,s',y,y'}\|Z(s\,,y)-Z(s',y')\| > 2^{-n-1}\}$.
	Recall the metric $\rho$ from \eqref{def:rho} and observe that
	$\sup_{s,s',y,y'}\rho((s\,,y)\,,(s',y')) \le 5\times 2^{-n(1+\delta)}.$ 
	Therefore, 
	 condition \emph{(i)} and Chebyshev's inequality
	together imply that there exist $L_2=L_2(\delta)>0$ and $L_3>0$ that
	depend only on $(c\,,\gamma\,,\mathcal{E}_0\,,\delta\,,L_2)$ and satisfy
	\begin{align*}
		\mathscr{P}_n &\le \P\left\{ \sup_{\substack{(s,y),(s',y')\in
			[S,T]\times\T\\(s,y)\ne(s',y')}}
			\frac{\|Z(s\,,y)-Z(s',y')\|}{\rho((s\,,y)\,,(s',y'))
			\sqrt{\log_+(1/\rho((s\,,y)\,,(s',y')))}} \ge \frac{2^{\delta n}}{
			\sqrt{L_2 n}}\right\}\\
		&\le \mathcal{E}_0 \exp\left( - c \left[\frac{2^{\delta n}}{\sqrt{L_2 n}}\right]^\gamma\, \right)
			\le L_3^n\e^{- n^2/L_3},
	\end{align*}
	uniformly for all $n\in\N$.
	Combine this with \eqref{E:p1p2:F(n)} to conclude the proof.
\end{proof}

The preceding lemma is key to deriving lower bounds for the Hausdorff dimension of
random sets of the form $Z(\{s\}\times F)$, valid uniformly in $s\in[0\,,T]$
and compact sets $F\subset\T$.
We now turn to sufficient conditions for deriving
lower bounds for the anisotropic Hausdorff dimension $\dim^\rho_{_{\rm H}}$ of random sets of the form $Z(G)$,
valid uniformly for all compact $G\subset[0\,,T]\times\T$;
see \eqref{E:dim:rho} and the subsequent paragraph. In fact, we state 
and prove a more general
result since we anticipate  future applications in other contexts.

\begin{lemma}\label{lem:unif:dim2}
	Consider an $\R^p$-valued random field $\{ Z(s) \}_{s\in I}$
	where $I\subset\R^N$ is a non-random, compact, upright
	rectangle. Choose and fix $\alpha_1, \dots, \alpha_N \in (0\,,1]$, and 
	metrize $I$ using
	$\mathsf{d}(s\,,s') = \sum_{j=1}^N |s_j-s'_j|^{\alpha_j}$
	for all $s,s'\in I.$
	For every $n \in \N$, $\delta \in (0\,,1)$, $\nu \in \R^p$, 
	and $r > 0$, define
	\begin{align}\begin{split}\label{E:F_n2}
		&\mathcal{N}_n^\delta(\nu\,, r) = \# \{ s \in F_n^\delta : Z(s) \in B(\nu\,, r) \}, \text{ where}\\
		&F_n^\delta = I \cap \bigcup_{j_1,\ldots,j_N\in\Z}
		\left\{ (j_1 2^{-(1+\delta)n/\alpha_1}, 
		\dots, j_N 2^{-(1+\delta)n/\alpha_N} )\right\},
	\end{split}\end{align}
	and suppose that the following two conditions hold:
	\begin{compactenum}[(i)]
	\item For every $\varepsilon\in(0\,,1)$,
		$\sup_{s,s'\in I:s\neq s'}\| Z(s)-Z(s')\|/\mathsf{d}(s\,,s')^{1-\varepsilon}<\infty;$
	\item For every  $R>0$, $\lim_{\delta\to0+}\limsup_{n\to\infty} n^{-1} \sup_{\nu \in B(0, R)}
		\log \mathcal{N}_n^\delta(\nu\,, 2^{-n}) =0$ a.s.\
	\end{compactenum}
	Then, there exists a $\P$-null set off which
	$\dimh  Z(G) = \dim_{_{\rm H}}^\rho G$,
	uniformly for all compact sets $G \subset I$.
\end{lemma}

The proof of the above lemma requires making only minor adaptations to that of Lemma \ref{lem:interpolation}
and therefore is left to the interested reader. 
It might help to emphasize that the following lemma provides a sufficient condition for part
\emph{(ii)} of Lemma \ref{lem:unif:dim2} to be applicable.

\begin{lemma}\label{lem:interpolation2}
	Let $\{Z(s)\}_{s\in I}$ be an $\R^p$-valued random field, where 
	$I \subset \R^N$ is a non-random, compact, upright rectangle. 
	Let $\mathsf{d}$, $\mathcal{N}_n^\delta(\nu\,,r)$, and $F_n^\delta$ be as defined
	in Lemma \ref{lem:unif:dim2}, and suppose that the following two conditions hold:
	\begin{compactenum}[(i)]
		\item The following is finite for some $c>0$ and $\gamma>0$:
		\begin{align}\label{E:int:lem2:a1}
			\mathcal{E}_1 := \E\left[ \exp\left( c \sup_{\substack{s,s'\in I\\ s\ne s'}}
			\left[\frac{\|Z(s)-Z(s')\|}{\mathsf{d}(s\,,s')\left| 
			\log_+ (1/\mathsf{d}(s\,,s'))\right|^{1/2}}
			\right]^\gamma \right) \right];
		\end{align}
		\item For every $\delta\in(0\,,1)$, there exist $a,b>0$ and $\kappa>0$ such that
		\begin{align}\label{E:int:lem2:a2}
			\P\left\{ \mathcal{N}_n^\delta(\nu\,,2^{-n}) \ge 2^{an} \right\}
			\le b 2^{-\kappa n^2}
			\quad\text{$\forall n\in\N$, $s\in I$, $\nu \in \R^p$}.
		\end{align}
	\end{compactenum}
	Then, for every $\delta\in(0\,,1)$ and $R>0$, there exists 
	$L>0$ such that
	\[
		\P\left\{ \sup_{\nu \in B(0,R)} \mathcal{N}_n^\delta(\nu\,,2^{-n})
		\ge 2^{an} \right\} \le L^n \e^{- n^2/L}
		\quad\text{uniformly for all $n\in\N$}.
	\]
\end{lemma}

\begin{proof}
	Fix $\delta\in(0\,,1)$ and $R>0$.
	Similarly to the proof of Lemma \ref{lem:interpolation}, 
	we can use interpolation and \eqref{E:int:lem2:a2} to write
	\begin{align*}
		&\P\left\{ \sup_{\nu \in B(0,R)} 
            \mathcal{N}_n^\delta(\nu\,,2^{-n})
			\ge 2^{an} \right\}\\
		&\hskip.8in\le L_1^n \e^{-n^2/L_1} + 
			\P\left\{ \sup_{s, s'\in I: \mathsf{d}(s,s')\le 
			2^{-n(1+\delta)}} \|Z(s) - Z(s')\| > 2^{-n-1} \right\},
	\end{align*}
	for all $n\in\N$, where $L_1>0$ is a number depending only on $(p\,,\delta\,,R\,,\kappa)$.
	Then, we further use Chebyshev's inequality and \eqref{E:int:lem2:a1} to see that 
	there exist numbers $L_2=L_2(\delta)>0$ and $L_3=L_3(c\,,\gamma\,,\mathcal{E}_1\,,\delta\,,L_2)>0$
	such that
	\begin{align*}
		&\P\left\{ \sup_{s, s'\in I: \mathsf{d}(s,s')\le 2^{-n(1+\delta)}}
			\|Z(s) - Z(s')\| > 2^{-n-1} \right\}
			\\
		&\le \P\left\{ \sup_{\substack{s,s'\in I:\\ s\ne s'}}
			\frac{\|Z(s) - Z(s')\|}{\mathsf{d}(s\,,s')\sqrt{\log_+(1/\mathsf{d}(s\,,s'))}}
			\ge \frac{2^{\delta n}}{\sqrt{L_2 n}} \right\}
			\le \mathcal{E}_1 \exp\left( -c\left[ \frac{2^{\delta n}}{\sqrt{L_2n}}
			\right]^\gamma\right)\\
		&\le L_3^n \exp( -n^2/L_3)\quad\text{uniformly for all $n\in\N$}.
	\end{align*}
	This completes the proof.
\end{proof}

\section{Proof of Theorem \ref{th:HEAT:add}}
	We are now ready to prove Theorem \ref{th:HEAT:add}; this is our broadest
	uniform dimension result, available when the noise term in \eqref{SHE} is additive.
	
	The proof is divided in two parts. 
	In the first part, we prove Theorem \ref{th:HEAT:add} in the special case that $b \equiv 0$.
	In the second part, we extend the result to general non-random 
	Lipschitz continuous functions $b: \R^p \to \R^p$.\medskip

	\emph{Part 1.} Suppose $b \equiv 0$, and $\sigma$ is a constant nonsingular $p \times p$ matrix.
	We may write 
	\begin{align}\label{E:u:b=0}
		u(t\,,x) = (\mathcal{G}_t u_0)(x) + \sigma H(t\,,x),
	\end{align}
	where $H$ is the solution to \eqref{H} which is a centered Gaussian random field with i.i.d.~coordinate processes.\medskip
	
\emph{Proof of (i).} Suppose $p \ge 2$. In order to prove \emph{(i)}
let us fix $0 < S < T < \infty$, and set 
$Z(t\,,x) = u(t\,,x)$ for $(t\,,x) \in [S\,,T]\times \T$  and $\alpha=\frac12$ in Lemma \ref{lem:unif:dim}.
Let $\mathcal{N}_n^\delta(t\,,B(\nu\,, r))$ and $F_n^\delta$ be defined by 
\eqref{E:Z:N} with $\alpha = \frac1 2$, and define $\widetilde{\mathcal N}_n^\delta(t\,,B(\nu\,,r))$ 
to be the total number of all $n$-tuples of all
$x_1 < \cdots < x_n$ in $F_n^\delta$ such that $u(t\,,x_i) \in B(\nu\,,r)$ for all $1 \le i \le n$;
that is,
\begin{align*}
	\widetilde{\mathcal N}_n^\delta(t\,,B(\nu\,,r)) = 
	\mathop{\sum\cdots\sum}\limits_{x_1 < \cdots < x_n \, \text{in}\, F_n^\delta}
	\1_{\{\max_{1 \le i \le n}\|u(t,x_i)-\nu\|\le r\}}.
\end{align*}
For every $n \in \N$, $\delta \in (0\,,1)$, $t \in [S\,,T]$, and $\nu \in \R^p$, we have
\begin{align}\label{E:P:N}
	&\P\left\{ \mathcal{N}_n^\delta(t\,,B(\nu\,,2^{-n})) \ge 2^{2\delta pn} \right\}\\
	\nonumber
	&
		\le \P\left\{ \widetilde{\mathcal N}_n^\delta(t\,,B(\nu\,,2^{-n})) 
		\ge \binom{\lceil 2^{2\delta pn}\rceil}{n}\right\}
		\le \binom{\lceil 2^{2\delta pn}\rceil}{n}^{-1} 
		\E\left[ \widetilde{\mathcal N}_n^\delta(t\,,B(\nu\,,2^{-n})) \right],
\end{align}
where the second inequality holds due to Chebyshev's inequality, and the expectation is given by
\[
	\E\left[ \widetilde{\mathcal N}_n^\delta(t\,,B(\nu\,,2^{-n})) \right] 
	= \mathop{\sum\cdots\sum}\limits_{x_1 < \cdots < x_n\, \text{in}\,F_n^\delta}
	\P\left\{ \max_{1 \le i \le n}\|u(t\,,x_i)-\nu\|\le 2^{-n} \right\}.
\]
We estimate the last probability as follows.
For all  $x_1 < \cdots < x_n$ in $F_n^\delta$
and  $\varepsilon > 0$, we use \eqref{E:u:b=0} and apply the Anderson shifted-ball inequality 
\cite{An} in order to see that
\begin{align*}
	\P\left\{ \max_{1\le i \le n}\|u(t\,,x_i) - \nu\| \le \varepsilon \right\}
		&\le\P\left\{ \max_{1\le i \le n}\| \sigma H(t\,,x_i) \| \le \varepsilon \right\}\\
	&\le \left( \P\left\{ \max_{1\le i\le n}|H_1(t\,,x_i)| \le 
		\varepsilon/\lambda_\sigma \right\} \right)^p,
\end{align*}
where $\lambda_\sigma > 0$ denotes the smallest singular value of $\sigma$.
In the preceding, we have  used the fact that the coordinates of $H$ are i.i.d. 
Since the probability density function of a centered, real-valued Gaussian random variable $Z$
is bounded above by $1/\sqrt{2\pi\Var (Z)}$, we proceed iteratively by 
successive conditioning on the values of $H_1(t\,,x_1),\ldots,H_1(t\,,x_i)$,
as $i$ varies from $n-1$ to 1, in order to see that uniformly for all $\varepsilon > 0$, 
$n \in \N$, $t \in [S\,,T]$ and $\nu \in \R^p$,
\begin{align*}
	&\P\left\{ \max_{1\le i \le n}\|u(t\,,x_i) - \nu\| \le \varepsilon \right\}\\
	&\le \left( \frac{2\varepsilon}{\lambda_\sigma \sqrt{2\pi \Var(H_1(t\,,x_1))}}
		\right)^p \prod_{i=2}^n \left( \frac{2\varepsilon}{\lambda_\sigma
		\sqrt{2\pi \Var(H_1(t\,,x_i) \mid \mathscr{H}_{i-1})}} \right)^p\\
	&\le C^n \varepsilon^{pn} t^{-p/4} \prod_{i=2}^n |x_i-x_{i-1}|^{-p/2},
\end{align*}
where $\mathscr{H}_i$ denotes the $\sigma$-algebra generated 
by $H_1(t\,,x_1), \dots, H_1(t\,,x_i)$. The last inequality follows from 
Lemma \ref{lem:Var(H)} and strong local nondeterminism in the form
of Proposition \ref{pr:LND} since $\sqrt{t}\geq \sqrt{S} \gtrsim |x_i-x_{i-1}|$ for all $i=2, \ldots, n$, 
and the number $C>0$ depends only on $(p\,,S\,,T\,,\sigma)$.
Whenever $x_0 \in F_n^\delta$,
\[
	\sum_{x \in F_n^\delta \setminus\{x_0\}} |x-x_0|^{-p/2}
	\le \sum_{1 \le j \le 2^{2(1+\delta)n}} \frac{2}{(j2^{-2(1+\delta)n})^{p/2}}.
\]
Therefore, the integral test of calculus implies that  
\[
	\sum_{x \in F_n^\delta \setminus\{x_0\}} |x-x_0|^{-p/2}\lesssim C_p(n)=
	\begin{cases}
		2^{p(1+\delta)n} & \text{if } p \ge 3,\\
		n 2^{2(1+\delta)n} & \text{if } p = 2,
	\end{cases}
\]
where the implied constant depends only on $(p\,,S\,,T\,,\sigma)$.
An iterative application of the preceding, applied with $\varepsilon = 2^{-n}$, implies that
there exist constants $c_i=c_i(p\,,S\,,T\,,\sigma)>0$ $[i=1,2]$ such that,
uniformly for all $n \in \N$, $t \in [S\,,T]$ and $\nu \in \R^p$,
\begin{align*}
	&\E\left[ \widetilde{\mathcal N}_n^\delta(t\,,B(\nu\,,2^{-n})) \right]\\
	& \le c_1^n 2^{-pn^2} \mathop{\sum}\limits_{x_1 \in F_n^\delta} S^{-p/4}
		\mathop{\sum}\limits_{x_2 \in F_n^\delta \setminus\{x_1\}} |x_2-x_1|^{-p/2}\,
		\cdots\hskip-.3in \mathop{\sum}\limits_{x_n \in F_n^\delta \setminus\{x_{n-1}\}}
		|x_n-x_{n-1}|^{-p/2}\\
	& \le c_2^n 2^{-pn^2} [C_p(n)]^n.
\end{align*}
Plug this inequality into \eqref{E:P:N} and use Lemma \ref{lem:stirling} to see that
\begin{align}\label{E:add:PN1}
	\P\left\{ \mathcal{N}_n^\delta(t\,,B(\nu\,,2^{-n})) \ge 2^{2\delta pn} \right\} \le \begin{cases}
		c_3^n n^n 2^{-\delta p n^2} & \text{if } p \ge 3,\\
		c_3^n n^{2n} 2^{-\delta p n^2} & \text{if } p = 2,
	\end{cases}
\end{align}
for some $c_3>0$.
Thanks to Lemmas \ref{lem:u:mod} and \ref{lem:interpolation}, 
for every $\delta\in(0\,,1)$ and $R>0$, there exist $K=K(\delta)>0$ 
and $L=L(p\,,S\,,T\,,\sigma\,,\delta\,,R)>0$ such that
\[
	\P\left\{ \sup_{t \in [S,T]} \sup_{\nu \in B(0, R)} 
	\mathcal{N}_n^\delta(t\,,B(\nu\,,2^{-n})) \ge K2^{2\delta p n} \right\} 
	\le L^n \e^{- n^2/L}\quad\forall n\in\N.
\]
Therefore, the Borel--Cantelli lemma implies that, almost surely,
\begin{align}\label{E:add:N1}
	\sup_{t\in[S,T]} \sup_{\nu \in B(0,R)} \mathcal{N}_n^\delta(t\,,B(\nu\,,2^{-n})) 
	=  \mathcal{O}(2^{2\delta p n}) \quad \text{as } n \to \infty.
\end{align}
Moreover, Lemma \ref{lem:u:mod} implies that a.s.,
\begin{align}\label{E:add:mod1}
	\sup_{t \in [S,T]} \sup_{\substack{x, x' \in \T:\\x\ne x'}}
	\frac{\|u(t\,,x)-u(t\,,x')\|}{|x-x'|^{(1-\varepsilon)/2}} < \infty
	\quad\forall\varepsilon \in (0\,,1).
\end{align}
With \eqref{E:add:N1} and \eqref{E:add:mod1} in place, 
we can apply Lemma \ref{lem:unif:dim} to deduce that
\[
	\dimh u(\{t\}\times F) = 
	2 \dimh F \quad\forall \text{compact $F \subset \T,\ t \in [S\,,T]$},
\]
off a single $\P$-null set.
Part \emph{(i)} follows  since $S>0$ and $T>0$ are arbitrary.\medskip
	
\emph{Proof of (ii).} Suppose $p \ge 4$. The proof of
\emph{(ii)} is similar to that of case \emph{(i)} above. 
    Fix $0<S<T$ and set $\alpha = \frac14$.
	For any $n \in N$, $\delta \in (0\,,1)$, $x \in \T$, $\nu \in \R^p$ and $r>0$, define
    \begin{align*}
        &\mathcal{N}^\delta_n(x\,,B(\nu\,,r)) = \# \{ s \in F^\delta_n : u(s\,,x) \in B(\nu\,,r) \},\\
        &\text{where} \quad 
        F^\delta_n = [S\,,T] \cap \{j2^{-n(1+\delta)/\alpha}\}.
    \end{align*}
    Let $\widetilde{\mathcal N}_n^\delta(x\,,B(\nu\,,r))$ denote 
	total the number of $n$-tuples $t_1 < \cdots < t_n$ in $F_n^\delta$ 
	such that $u(t_i\,,x) \in B(\nu\,,r)$ for all $1 \le i \le n$.
	Then, as in the proof of \eqref{E:add:PN1}, we can use strong local 
	nondeterminism (Proposition \ref{pr:LND}) to deduce that for 
	every $n\in \mathbb{N}$, $\delta \in (0\,,1)$, $x \in \T$, and $\nu\in \R^p$,
	\begin{align*}
		&
			\P\left\{ \mathcal{N}_n^\delta(x\,,B(\nu\,,2^{-n})) 
			\ge 2^{2\delta p n} \right\}\\
		&\le \binom{\lceil 2^{2\delta p n}\rceil}{n}^{-1}
			\sum_{t_1 < \cdots< t_n \, \text{in}\, F^\delta_n}
			\P\left\{ \max_{1 \le i \le n}\|u(t_i\,,x)-\nu\| \le 2^{-n} \right\}\\
		&\le C_1^n n^n 2^{-2\delta pn^2-pn^2} \mathop{\sum}\limits_{t_1 \in F^\delta_n}
			S^{-p/4} \mathop{\sum}\limits_{t_2 \in F_n^\delta \setminus\{t_1\}} 
			|t_2-t_1|^{-p/4}\dots \hskip-0.2in
			\mathop{\sum}\limits_{t_n \in F_n^\delta \setminus\{t_{n-1}\}} |t_n-t_{n-1}|^{-p/4}\\
		& \le \begin{cases}
			C_2^n n^n 2^{-\delta pn^2} & \text{if } p \ge 5,\\
			C_2^n n^{2n} 2^{-\delta pn^2} & \text{if } p = 4,
		\end{cases}
	\end{align*}
	for some $C_i=C_i(p\,,S\,,T\,,\sigma)>0$ $[i=1,2]$. Thanks to
	Lemmas \ref{lem:u:mod} and \ref{lem:interpolation} and the Borel--Cantelli lemma, 
	the preceding implies that for every $\delta \in (0\,,1)$ and $R>0$,
	\[
		\sup_{x \in \T} \sup_{\nu \in B(0,R)} \mathcal{N}_n^\delta(x\,,B(\nu\,,2^{-n})) = 
		\mathcal{O}(2^{2\delta p n}) \quad \text{as } n \to \infty\text{ a.s.}
	\]
	Also, by Lemma \ref{lem:u:mod}, there exists a $\P$-null set off which
	\[
		\sup_{x \in \T} \sup_{\substack{t, t' \in [S, T]:\\t \ne t'}}
		\frac{\|u(t\,,x) - u(t',x)\|}{|t-t'|^{(1-\varepsilon)/4}} < \infty \quad
		\forall\varepsilon \in (0\,,1).
	\]
    Thanks to the preceding two displays, we may prove in the same way as in Lemma \ref{lem:unif:dim} (with $\alpha = 1/4$) -- except that the roles of $s$ and $y$ are now reversed -- that there is a single $\P$-null set off which
	\[
		\dimh u(F\times\{x\}) = 4 \dimh F \quad \forall \text{compact $F \subset [S\,,T],\
		 x\in\T$.}
	\]
	Since $T>S>0$ are arbitrary non-random numbers, 
	this proves part \emph{(ii)}.\medskip
	
	\emph{Proof of (iii).} Suppose that $p \ge 6$.  Fix $0 \le S < T$. 
	We will apply Lemma \ref{lem:unif:dim2} with 
	$I = [S\,,T]\times\T \cong [S\,,T] \times [-1\,,1]$, $\alpha_1 = 1/4$, $\alpha_2 = 1/2$,
	and $Z(s) = u(s)$ for all $s\in I$.
	Let $\mathcal{N}_n^\delta(\nu\,,r)$ and $F_n^\delta$ as be defined in \eqref{E:F_n2}.
	In particular, 
	\[
	F_n^\delta = I \cap \{ (j_1 2^{-4(1+\delta)n}, j_2 2^{-2(1+\delta)n}) : j_1, j_2 \in \Z \}.
	\]
	In order to establish part \emph{(iii)}, we adopt an idea of 
	Monrad and Pitt \cite{MP87}.
	The key observation is that every set of $n$ distinct points 
	$s_1, \dots, s_n \in [S\,,T] \times [-1\,,1]$ can be reordered so that 
	\begin{align}\label{E:ord}
		\rho(s_i\,,s_{i-1}) \le \rho(s_i\,,s_j) \quad \forall 1 \le j < i \le n.
	\end{align}
	This can be done in at least $n$ different ways -- we pick in any manner we like
	a point as $s_n$ first, then find $s_{n-1}$, then $s_{n-2}$, etc.
	By the Anderson shifted-ball inequality \cite{An} and strong local nondeterminism (Proposition \ref{pr:LND}),
	\begin{align*}
		&\P\left\{ \max_{1\le i \le n} \|u(s_i) - \nu \| \le \varepsilon \right\}
			\le \P\left\{ \max_{1\le i\le n}|H_1(s_i)| \le 
			\varepsilon/\lambda_\sigma \right\}^p\\
		& \le \left( \frac{2\varepsilon}{\lambda_\sigma \sqrt{2\pi \Var(H_1(s_1))}}
			\right)^p \prod_{i=2}^n \left( \frac{2\varepsilon}{\lambda_\sigma 
			\sqrt{2\pi\Var(H_1(s_i) \mid \mathscr{H}_{i-1})}} \right)^p\\
		& \le C^n \varepsilon^{pn} S^{-p/4}
			\prod_{i=2}^n \left[ \min\limits_{1 \le j \le i-1}\rho(s_i\,,s_j)\right]^{-p}
		\le C^n \varepsilon^{pn} S^{-p/4}\prod_{i=2}^n 
			\rho(s_i\,,s_{i-1})^{-p},
	\end{align*}
	uniformly for all $\varepsilon > 0$, $n \in \N$ and $\nu \in \R^p$, and for
	all $n$ distinct points $s_1, \dots, s_n \in F_n^\delta$ that satisfy \eqref{E:ord}.
	Here, $\mathscr{H}_i$ denotes the $\sigma$-algebra generated by 
	$H_1(s_1), \dots, H_1(s_{i})$, and $C=C(p\,,S\,,T\,,\sigma)>0$ is a fixed number.
	Whenever $s_0 \in F_n^\delta$,
	\[
		\mathop{\sum}\limits_{s \in F_n^\delta \setminus\{s_0\}} \rho(s\,,s_0)^{-p} 
		\le \hskip-.2in
		\mathop{\sum}\limits_{\substack{j \in \Z^2 \setminus \{0\}:\\ 
		\rho(j,0) \le C2^{(1+\delta)n}}} \frac{2^{p(1+\delta)n}}{(\rho(j\,,0))^p} 
		\lesssim 2^{p(1+\delta)n} \int_1^{C2^{(1+\delta)n}} 
		\hskip-.2in     r^{5-p}\, \d r
		\le C_p(n),
	\]
	where
	\[
		C_p(n) = \begin{cases}
		C 2^{p(1+\delta)n} & \text{if } p \ge 7,\\
		C n 2^{6(1+\delta)n} & \text{if } p = 6.
		\end{cases}
	\]
	Let ${\mathcal M}_n^\delta(\nu\,,r)$ denote the total
	number of all $n$-tuples of distinct points $s_1, \dots, s_n$
	in  $F_n^\delta$ that satisfy \eqref{E:ord} and $u(s_i) \in B(\nu\,,r)$ for all $1 \le i \le n$.
	It follows that
	\begin{align*}
		&
			\P\left\{ \mathcal{N}_n^\delta(\nu\,,2^{-n}) \ge 2^{2\delta pn} \right\}
			\le \P\left\{ {\mathcal M}_n^\delta(\nu\,,2^{-n}) \ge  
			\binom{\lceil 2^{2\delta pn}\rceil}{n} \right\}\\
		& 
			\le \binom{\lceil 2^{2\delta pn}\rceil}{n}^{-1}  
			\sum_{\substack{\text{distinct } s_1, \dots, s_n \, \text{in}\,F_n^\delta\\ 
			\text{that satisfy } \eqref{E:ord}}}
			\P\left\{ \max_{1\le i \le n} \|u(s_i) - \nu \| \le 2^{-n} \right\}\\
		& \le C^n n^n 2^{-2\delta pn^2-pn^2} \mathop{\sum}\limits_{s_1 \in F_n^\delta} 
			S^{-p/4} \mathop{\sum}\limits_{s_2\in F_n^\delta \setminus\{s_1\}} 
			\rho(s_2\,,s_1)^{-p}\, \cdots \hskip-0.3in
			\mathop{\sum}\limits_{s_n \in F_n^\delta \setminus\{s_{n-1}\}}
			\rho(s_n\,,s_{n-1})^{-p}\\
		& \le \begin{cases}
			c^n n^n 2^{-\delta pn^2} & \text{if } p \ge 7,\\
			c^n n^{2n} 2^{-\delta pn^2}& \text{if } p = 6,
		\end{cases}
	\end{align*}
	where $c=c(p\,,S\,,T\,,\sigma)>0$.
	By Lemmas \ref{lem:u:mod} and \ref{lem:interpolation2}, and the Borel--Cantelli lemma, 
	for every $\delta \in (0\,,1)$ and $R>0$,
	$\sup_{\nu \in B(0,R)} \mathcal{N}_n^\delta(\nu\,,2^{-n}) = 
	\mathcal{O}(2^{2\delta pn})$ as $n \to \infty$, almost surely.
	Also, by Lemma \ref{lem:u:mod}, a.s.,
	\[
		\sup_{s, s' \in [S,T]\times\T:\, s \ne s'} 
		\frac{\|u(s) - u(s')\|}{(\rho(s\,,s'))^{1-\varepsilon}} < \infty \quad 
		\forall\varepsilon \in (0\,,1).
	\]
	Thanks to the preceding two displays, we can apply Lemma 
	\ref{lem:unif:dim2} in order to see that there exists a $\P$-null set off which
	\[
		\dimh u(G) = \dim_{_{\rm H}}^\rho G \quad 
		\forall\text{compact $G \subset [S\,,T]\times\T$}.
	\]
	Since $S>0$ and $T>0$ are arbitrary,  part \emph{(iii)} follows.\medskip
	
	\emph{Part 2.} Now, we consider the general case that the non-random function $b:\R^p \to \R^p$ is
	Lipschitz continuous.  Define
	\begin{equation}\label{b_N}
		b_N(x) =
		\begin{cases}
			b(x) & \text{if } \|x\| \le N,\\
			b(Nx/\|x\|) & \text{if } \|x\| > N.
		\end{cases}
	\end{equation}
	Then it is not hard to see that $b_N$ is globally Lipschitz.
	In fact, 
	\begin{equation}\label{E:liplip:b}
		\lip(b_N)\le \lip(b)\quad\forall N>0.
	\end{equation}
	Here is the short proof: 
	If $b\in C^1_b(\R^p)$, then for all $N>0$ we have
	$\partial_{v_i}b_N(v)=\partial_{v_i}b(v)$ when $\|v\|\le N$ and
	\[
		\partial_{v_i}b_N(v) = \partial_{v_i}b\left(\frac{vN}{\|v\|}\right)
		\frac{N}{\|v\|}\left( 1 - \frac{v_i^2}{\|v\|^2} \right)\quad\text{when $\|v\|> N$}.
	\]
	It follows that $\lip(b_N)=\|\nabla b_N\|_{L^\infty(\R^p)}\le
	\|\nabla b\|_{L^\infty(\R^p)}=\lip(b)$ for every $N>0$. 
	In general, when we know only that
	$b\in\lip(\R^p)$, we write
	$b^t=\phi_t*b$ where $\phi_t$ was defined in \eqref{G:phi}. Direct inspection shows
	that $\lip(b^t)\le\lip(b)$, whence $\lip(b^t_N)\le\lip(b)$ thanks to the
	preceding argument. This means, among other things, that
	$\|b^t_N(x)-b^t_N(y)\|\le\lip(b)\|x-y\|$
	for every $t>0$ and $x,y\in\R^p$. Send $t\downarrow0$
	to obtain \eqref{E:liplip:b}.
	
	Let $u_N$ denote the solution to \eqref{SHE} where $b$ is replaced by $b_N$ and $\sigma$ is a constant nonsingular $p\times p$ matrix.
	This equation can be rewritten as
	\begin{equation*}\left[\begin{aligned}
	&\partial_t u_N(t\,,x) = \partial^2_xu_N(t\,,x) + \sigma\left[ \sigma^{-1}b_N(u_N(t\,,x))
		+ \xi(t\,,x) \right]
			\quad\text{on $(0\,,\infty)\times\T$},\\
	&u_N(0) =u_0\quad\text{on }\T.
	\end{aligned}\right.\end{equation*}
	In other words, the mild formulation for the solution can be rewritten as
	\[
		u_N(t\,,x) = (\mathcal{G}_t u_0)(x) + \int_{(0, t)\times\T} 
		G_{t-s}(x\,,y)\, \sigma \left[ \sigma^{-1} b_N(u_N(s\,,y)) \, \d s\, \d y + \xi(\d s\, \d y) \right].
	\]
	Choose and fix $T>0$.
	Since $b_N$ is bounded and $\sigma$ is nonsingular, 
	Girsanov's theorem (see Lemma \ref{lem:girsanov}) implies
	that $\zeta(t\,,x) = \sigma^{-1} b_N(u_N(t\,,x)) + \xi(t\,,x)$
	$[(t\,,x) \in [0\,,T]\times \T]$ is a space-time white noise 
	on the probability space $(\Omega\,, \sF_T, \rQ)$, where $\rQ$ is mutually
	absolutely continuous with respect to $\P$. Under the measure $\rQ$, 	the random field $u_N$ solves
	$\partial_t u_N = \partial^2_xu_N + \sigma\zeta$
	on $(0\,,T)\times\T$ subject to $u_N(0) =u_0$ on $\T$.
	Therefore, Part 1 of this proof and the mutual absolute continuity of $\P$ and $\rQ$ 
	together yield
	\[
		\P(\Omega_i(u_N,T)) 
		=\rQ(\Omega_i(u_N,T)) = 1 \quad \text{if } p \ge p_i \quad [i=1,2,3],
	\]
	where $p_1 = 2$, $p_2 = 4$, $p_3 = 6$, and the $\Omega_i$s are
	the following three events:
	\begin{align*}
		&\Omega_1(u_N,T)= \left\{ \dimh u_N(\{t\}\times A) = 2\dimh A \
			\forall \text{compact  $A\subset\T, t\in(0\,,T)$}\right\},\\
		&\Omega_2(u_N,T)=\left\{ \dimh u_N(B\times\{x\}) = 4\dimh B\
			\forall \text{compact $B\subset(0\,,T),x\in\T$}\right\},\\
		&\Omega_3(u_N,T)=\left\{ \dimh u_N(C) = \dim_{_{\rm H}}^\rho\! C\
		\forall\text{compact }C\subset(0\,,T)\times\T\right\}.
	\end{align*}
	Since $T>0$ is arbitrary, we may let $T\to\infty$ to see that
	$\P(\Omega_i(u_N,\infty)) = 1$ if $p \ge p_i$  $[i=1,2,3].$\footnote{The
	proof of this included showing that the $\Omega_i$s include measurable
	sets of $\P$-mass one. Therefore, the $\Omega_i$s are themselves
	measurable, thanks to the completeness of the underlying probability space.}
	Define 
	\[
		T_N = \inf\{ t > 0: \sup_{x\in\T}\|u_N(t\,,x)\| \ge N \},
	\]
	where $\inf\varnothing = \infty.$
	Every $T_N$ is a stopping time with respect to the filtration $\{\sF_t\}_{t\ge0}$ 
	of the noise $\xi$, and the uniqueness of the solution to \eqref{SHE} implies that
	\[
		\P\{ u_N(t) = u(t) \ \forall t\in(0\,,T_N) \} = 1.
	\]
	It follows that
	$\P(\Omega_i(u\,,T_N)) = 1$ if $p \ge p_i$  $[i=1,2,3]$, notation being clear
	from context.
	By continuity, $u$ is bounded on $[0\,,t]\times\T$
	for every $t>0$, and hence $\lim_{N\to\infty}\P\{T_N>t\} = 1$ for every $t>0$.
	Therefore, we may let $N\to\infty$ in order to see that
	$\P(\Omega_i(u\,,\infty)) = 1$ if $p \ge p_i$ $[i=1,2,3].$
	This concludes the proof of Theorem \ref{th:HEAT:add}.
\qed

\section{Proof of Theorem \ref{th:HEAT:torus}}

For each $n \in \Z_+$, $T>0$, $t \in (0\,, T]$, $x \in \T$, 
$\nu \in \R^p$, $r > 0$, and $\delta > 0$, define
\[
	\mathcal{N}_n^\delta(t\,, B(\nu\,, r)) = \#\{ x \in F_n^\delta : u(t\,, x) \in B(\nu\,, r) \},
\]
where $F_n^\delta = \{ j 2^{-2(1+\delta)n} \in \T : j \in \Z \cap [- 2^{2(1+\delta)n}, 2^{2(1+\delta)n}] \}$.

\begin{proposition}\label{pr:key}
	Suppose $p \ge 4$, $\mathcal{M}(b) < \infty$, $\mathcal{M}(\sigma)<\infty$ 
	and $\inf_{v\in\R^p}\lambda(v)>0$.
	Fix $T > 0$ and $\delta \in (0\,,1)$.
	Then, there exists $L = L(p\,,\delta\,, T\,, b\,,\sigma\,, u_0) > 0$ such that
	\[
		\P\left\{\mathcal{N}^\delta_n(t\,, B(\nu\,, 2^{-n})) \ge 2^{2np\delta}\right\} 
		\le L^n n^{3pn} 2^{-\delta pn^2},
	\]
	uniformly for all $\nu \in \R^p$, $t \in (0\,, T]$ and $n \in \Z_+$ 
	such that $2^{-n} \le t/2$.
\end{proposition}

\begin{proof}
	Let $u_0 = h$ as in \S\ref{s:linearization}.
	Recall the definition of $I(t\,,x)$ in \eqref{I(t,x)} and 
	consider the $p$-dimensional random field
	\[
		\mathscr{E}(t\,,x) = I(t\,,x) - \sigma(h(x))H(t\,,x)
		\qquad\forall t\in(0\,,T], x\in\T.
	\]
	First, we claim that if $C_0$ is a constant such that 
	\begin{align}\label{E:H_k(h)}
		\mathcal{H}_k(h) \le C_0 \sqrt{k} \quad\forall k \in [2\,,\infty),
	\end{align}
	then there exists $L_1=L_1(p\,,T\,,b\,,\sigma\,,C_0)>0$ such that 
	\begin{equation}\label{E:III}
		\E\left( 
		\sup_{x\in\T}\| I(t\,,x) - \sigma(h(x))H(t\,,x)\|^k\right) \le L_1^k k^{3k/2} 
		t^{k/2}\left|\log_+(1/t)\right|^{3k/2},
	\end{equation}
	uniformly for all $t\in(0\,,T]$ and $k\in[2\,,\infty)$.
	To see why this is the case, we first apply Lemma \ref{lem:I-sX} and \eqref{E:H_k(h)} to
	obtain
	\[
		\E\left( \| \mathscr{E}(t\,,x)\|^k\right) \le C^k k^{3k/2} \,t^{k/2} \quad \forall k \in [2\,,\infty).
	\]
	This, together with Stirling's formula, implies 
	the existence of some $c_1 > 0$ such that
	\begin{align}\label{E:E:exp}
		\sup_{t,x}\E\,\e^{%
		c_1 \left[ \|\mathscr{E}(t\,,x)\|/\sqrt{t}\right]^{2/3}} \le 
		\sum_{k=0}^\infty \frac{c_1^k}{k!}\sup_{t,x}
		\frac{\E(\|\mathscr{E}(t\,,x)\|^{2k/3})}{t^{k/3}} < \infty,
	\end{align}
	where ``$\sup_{t,x}:=\sup_{(t,x)\in[0,T]\times\T}$'' on both sides of the above.
	Next, we write
	\[
		\mathcal{E}(t\,,x) = \mathcal{E}_1(t\,,x)+\mathcal{E}_2(t\,,x)-\mathcal{E}_3(t\,,x),
	\]
	where
	\begin{align*}
		& \mathcal{E}_1(t\,,x) = \int_{(0,t)\times\T} G_{t-s}(x\,,y) b(u(s\,,y)) \, \d s\, \d y,\\
		& \mathcal{E}_2(t\,,x) = \int_{(0,t)\times\T} G_{t-s}(x\,,y) \sigma(u(s\,,y)) \xi(\d s\, \d y),
			\qquad
			\mathcal{E}_3(t\,,x) = \sigma(h(x)) H(t\,,x).
	\end{align*}
	Recall the metric $\rho$ defined in \eqref{def:rho}.
	Since $\mathcal{M}(b)<\infty$, we may deduce as in the proofs 
	of Lemmas \ref{lem:u-u:x} and \ref{lem:u-u:t} that
	\begin{align*}
		\|\mathcal{E}_1(t\,,x) - \mathcal{E}_1(r\,,z)\|_k \lesssim
		\rho((t\,,x)\,,(r\,,z)) \quad \forall k \in [2\,,\infty),t,r \in (0\,,T],x,z \in \T.
	\end{align*}
	Similarly, since $\mathcal{M}(\sigma)<\infty$, we have
	\begin{align*}
		\|\mathcal{E}_2(t\,,x) - \mathcal{E}_2(r\,,z)\|_k \lesssim 
		\sqrt{k} \rho((t\,,x)\,,(r\,,z)) \,\,\, \forall k \in [2\,,\infty),t,r \in (0\,,T],x,z \in \T.
	\end{align*}
	Moreover, by the Gaussianity of $H$, Lemmas \ref{lem:G-G:x}, 
	\ref{lem:G-G:t}, \ref{lem:Var(H)}, and \eqref{E:H_k(h)},
	\begin{align*}
		&\|\mathcal{E}_3(t\,,x) - \mathcal{E}_3(r\,,z)\|_k \\
		&\le \|\sigma(h(x))(H(t\,,x)-H(r\,,z))\|_k + \|(\sigma(h(x))-\sigma(h(z)))H(r\,,z)\|_k\\
		&\lesssim \mathcal{M}(\sigma) \sqrt{k}\, \|H(t\,,x) - H(r\,,z)\|_2 
			+ \lip(\sigma) \mathcal{H}_k(h) |x-z|^{1/2} \sqrt{k}\, \|H(r\,,z)\|_2\\
		&\lesssim k \rho((t\,,x)\,,(r\,,z))
	\end{align*}
	uniformly for all $k \in [2\,,\infty)$, $t,r \in (0\,,T]$, $x,z \in \T$.
	Hence, we have
	\begin{align*}
		\|\mathcal{E}(t\,,x) - \mathcal{E}(r\,,z)\|_k \lesssim k 
		\rho((t\,,x)\,,(r\,,z)) \,\,\, \forall k \in [2\,,\infty), t,r\in (0\,,T], x,z \in \T.
	\end{align*}
	The preceding, together with a standard metric entropy argument, then yields the existence of a constant
	$c_2> 0$ such that
	\begin{align}\label{E:incr:exp}
		\E\exp\left(c_2 \sup_{(t,x),(s,y)} \frac{%
		\|\mathscr{E}(t\,,x)-\mathscr{E}(s\,,y)\|}{\rho((t\,,x)\,,(s\,,y)) 
		{\log_+(1/\rho((t\,,x)\,,(s\,,y)))}} \right)<\infty,
	\end{align}
	where ``$\sup_{(t,x),(s,y)}:=
	\sup_{(t,x),(s,y)\in[0,T]\times\T: (t,x)\neq(s,y)}.$''

	For every $m \in \N$, define $\T_m = \{ i/m \in \T : i \in \Z \cap [-m,m] \}$. 
	Let $\lambda_0\geq \e$ be a sufficiently large integer such that $a\mapsto a\log_+(1/a)$ is increasing on $(0\,, 1/\lambda_0]$.	
	Then, we argue by interpolation, and use \eqref{E:E:exp} and \eqref{E:incr:exp} together 
	with Chebyshev's inequality, in order to deduce that
	\begin{align*}
		&\P\left\{ \sup_{x \in \T} \|\mathscr{E}(t\,,x)\| > z \right\}\\
		& \le \P\left\{ \max_{x \in \T_m} \|\mathscr{E}(t\,,x)\| > z/2 \right\}
			+ \P\left\{ \sup_{x, y \in \T, |x-y|\le 1/m} \|\mathscr{E}(t\,,x)-
			\mathscr{E}(t\,,y)\| > z/2 \right\}\\
		&  \lesssim m \exp\left( - c_1 \left[ \frac{z}{2\sqrt t} \right]^{2/3} \right) + 
			\exp\left(-\frac{c_2 z}{2m^{-1/2}\log_+( m^{1/2})}\right),
	\end{align*}
	uniformly for all $t,z>0$ and $m\in\N$ with $m\geq \lambda_0^2$.
	Choose $m =\lambda_0^2 \lceil 1/t \rceil$ to find that there exists $c>0$ 
	such that for all $z > 0$ and $t \in (0\,,T]$,
	\begin{align*}
		\P\left\{ \sup_{x \in \T} \|\mathscr{E}(t\,,x)\| > z \right\}
		\lesssim \exp\left( \log_+(1/t) - \frac{cz^{2/3}}{t^{1/3}} \right) + 
		\exp\left( - \frac{cz}{{t^{1/2}\log_+(1/t)}} \right).
	\end{align*}
	It follows that
	\begin{align*}
		&\E\left( \sup_{x\in \T}\|I(t\,,x)-\sigma(h(x)) H(t\,,x)\|^k \right) 
			= \int_0^\infty k z^{k-1} \P \left\{ \sup_{x \in \T} 
			\|\mathscr{E}(t\,,x)\| > z \right\} \d z\\
		& \lesssim k 2^k c^{-3k/2} t^{k/2} \left|\log_+(1/t)\right|^{3k/2}\\
		&\quad  + 
			\int_{2c^{-3/2} t^{1/2} |\log_+(1/t)|^{3/2}}^\infty kz^{k-1} 
			\left(\e^{\log_+(1/t) - \frac{cz^{2/3}}{t^{1/3}}} + 
			\e^{- \frac{cz}{{t^{1/2}\log_+(1/t)}}} \right) \d z\\
		& \le k 2^k c^{-3k/2} t^{k/2} \left|\log_+(1/t)\right|^{3k/2} \\
		& \quad\, 
			\left[ 1 + \int_{1}^\infty y^{k-1} \left( \e^{\log_+(1/t) - 
			2^{2/3}\log_+(1/t) y^{2/3}} + \e^{-2c^{-1/2}[\log_+(1/t)]^{1/2} y} 
			\right) \d y\right]\\
		& \lesssim k 2^k c^{-3k/2} t^{k/2} \left|\log_+(1/t)\right|^{3k/2} 
			\left[ 1 + \int_1^\infty y^{k-1} \e^{-(2^{2/3}-1)y^{2/3}} 
			\, \d y \right]\\
		& \lesssim L^k t^{k/2} \left|\log_+(1/t)\right|^{3k/2} \Gamma(3k/2),
	\end{align*}
	uniformly for all $t\in(0\,,T]$ and $k\in[2\,,\infty)$, 
	where $L>0$ is a constant independent of $t$ and $k$.
	This proves \eqref{E:III} under condition \eqref{E:H_k(h)}.

	Next, let us write a subscript of $t$ as follows to simplify the notation:
	$u_t(x) = u(t\,,x).$
	This slightly abuses notation, since $u_1,\ldots,u_p$ represent the respective coordinates
	of $u$, but it is consistent with standard probability nomenclature. To be sure, if we ever need
	to refer to the $i$th coordinate of $u_t(x)$, then we would write $u_{t,i}(x)$.

	Consider a number $\eta\in(0\,,t)$, hold it fixed,
	and then apply the Markov property \cite{DZ}*{Chapter 9} at time $t-\eta$ in order to see that
	the mild formulation \eqref{mild} of the solution can be written as follows:
	\begin{equation}\label{E:MP}
		u_t(x) =  (\mathcal{G}_\eta u_{t-\eta})(x) + \tilde{I}(\eta\,,x),
	\end{equation}
	where
	\begin{align*}
		\tilde{I}(\eta\,,x) &
			= \int_{(0,\eta)\times\T} G_{\eta-s}(x\,,y) b(u_{t-\eta+s}(y)) \, \d s \, \d y\\
		&\hskip1in + \int_{(0, \eta) \times \T}
			G_{\eta-s}(x\,,y) \sigma(u_{t-\eta+s}(y))\, \xi^{(t-\eta)}(\d s\,\d y),
	\end{align*}
	and $\xi^{(a)}$ denotes a space-time white noise that is independent of $\sF_a$.
	In fact, $ \xi^{(a)}$ corresponds to a time shift by $a$ units in the noise's time variable. 
	Thanks to Lemma \ref{lem:u-u:x} and the assumption that $u_0 \in C^{1/2}(\T)$
	-- see the Introduction -- $\mathcal{H}_k(u_{t-\eta})\le C_0 \sqrt{k}$ 
	uniformly for all $0 < \eta < t \le T$ and $k \in [2\,, \infty)$, 
	where $C_0$ is a positive number that depends only on $u_0$.
	Therefore, we may apply \eqref{E:III},
	conditionally on $\sF_{t-\eta}$, in order to deduce  that
	\begin{align*}
		&
			\E\left( \sup_{x\in\T}\left\| \tilde{I}(\eta\,,x) - \sigma(u_{t-\eta}(x))\int_{(0, \eta) \times \T}
			G_{\eta-s}(x\,,y) \, \xi^{(t-\eta)}(\d s\,\d y) \right\|^k
			\right)\\
		&\hskip1in
			\le L_1^k k^{3k/2}\eta^{k/2}\left|\log_+(1/\eta)\right|^{3k/2},
	\end{align*}
	where $L_1=L_1 (p,\,T\,,b\,,\sigma\,, u_0) >0$.
	This and \eqref{E:MP} together yield
	\begin{align}\nonumber
		&\E\left( \sup_{x\in\T}\left\| u(t\,,x) - (\mathcal{G}_\eta u_{t-\eta})(x)  - 
			\sigma(u_{t-\eta}(x))\int_{(0, \eta) \times \T}
			G_{\eta-s}(x\,,y) \, \xi^{(t-\eta)}(\d s\,\d y) \right\|^k
			\right) \\
		&\hskip1in \le L_1^k k^{3k/2}\eta^{k/2}\left|\log_+(1/\eta)\right|^{3k/2},
			\label{E:Q}
	\end{align}
	once again with good parameter dependencies.
	Then, \eqref{E:Q} and Chebyshev's inequality imply that, for every $\varepsilon>0$,
	\begin{align}\nonumber
		&
			\P\left\{\sup_{x\in\T}\left\| u(t\,,x) - (\mathcal{G}_\eta u_{t-\eta})(x)  
			- \sigma(u_{t-\eta}(x))\int_{(0, \eta) \times \T} G_{\eta-s}(x\,,y) \, 
			\xi^{(t-\eta)}(\d s\,\d y) \right\| > \varepsilon\right\}\\
		&\hskip1in \le (L_1/\varepsilon)^k k^{3k/2}\eta^{k/2}
			\left|\log_+(1/\eta)\right|^{3k/2}.\label{E:P2}
	\end{align}
	Set ${\mathcal{N}}^\delta_n(v\,;t\,,\eta\,, B(\nu\,,r)) = 
	\#\{ x \in F_n^\delta : v(t\,,\eta\,,x) \in B(\nu\,, r) \}$, where
	\begin{align*}
		v(t\,,\eta\,, x) = (\mathcal{G}_\eta u_{t-\eta})(x) 
		+ \sigma(u_{t-\eta}(x)) \int_{(0, \eta) \times \T} G_{\eta-s}(x\,, y)\, \xi^{(t-\eta)}(\d s\, \d y),
	\end{align*}
	and choose $\varepsilon = \varepsilon_n = 2^{-n}$.
	By the  triangle inequality and \eqref{E:P2},
	\begin{align}\label{E:PN}
		&
			\P\left\{ \mathcal{N}^\delta_n(t\,, B(\nu\,,2^{-n})) \ge 2^{2\delta pn} \right\}\\\nonumber
		&
			\le \P\left\{ {\mathcal{N}}^\delta_n(v\,;t\,,\eta\,, B(\nu\,,2^{-n+1}))
			\ge 2^{2\delta pn} \right\} + (L_1/\varepsilon)^k k^{3k/2}
			\eta^{k/2}\left|\log_+(1/\eta)\right|^{3k/2},
	\end{align}
	uniformly for all $t \in (0\,, T]$, $\nu \in \R^p$, $n \in \Z_+$, 
	$\eta \in (0\,, t)$, and $k \in [2\,, \infty)$.
	Let $\widetilde{\mathcal{N}}^\delta_n(t\,,\eta\,, B(\nu\,,r))$ 
	denote the total number of $n$-tuples $x_1 < \cdots < x_n$ in $F_n^\delta$ 
	such that $v(t\,,\eta\,,x_i) \in B(\nu\,, r)$ for all $1\le i\le n$; namely,
	\begin{align*}
		\widetilde{\mathcal{N}}^\delta_n(t\,,\eta\,, B(\nu\,,r)) = 
		\mathop{\sum\cdots\sum}\limits_{x_1 < \cdots < x_n \, \text{in} \, F_n^\delta}
		\1_{\{ \max\limits_{1\le i \le n}\|v(t, \eta, x_i) - \nu\| \le r \}}.
	\end{align*}
	The above and Chebyshev's inequality together imply that
	\begin{align}\begin{split}
		&
			\P\left\{ {\mathcal{N}}^\delta_n(v\,;t\,,\eta\,, B(\nu\,,2^{-n+1}))
			\ge 2^{2\delta pn} \right\}\\
		&\le \P \left\{ \widetilde{\mathcal{N}}^\delta_n(t\,,\eta\,, B(\nu\,,2^{-n+1}))
			\ge \binom{\lceil 2^{2\delta pn} \rceil}{n}\right\}\\
		&\le
			\binom{\lceil 2^{2\delta pn}\rceil}{n}^{-1} 
			\mathop{\sum\cdots\sum}\limits_{x_1 < \cdots < x_n \, \text{in}\, F_n^\delta}
			\P\left\{ \max_{1\le i \le n}\|v(t, \eta, x_i) - \nu\| \le 2^{-n+1} \right\}.
			\label{E:PM}
	\end{split}\end{align}
	In order to estimate the last probability, let us consider an arbitrary 
	but fixed $n$-tuple of distinct points 
	$x_1 < \cdots < x_n$ in $F_n^\delta$, condition on $\sF_{t-\eta}$, and 
	notice that the quantity inside the $\|\,\cdots\|$ in the event is conditionally 
	a centered and continuous Gaussian process. 
	Therefore, we may
	apply conditionally the Anderson shifted-ball inequality \cite{An}
	in order to see that
	\begin{align*}
		&
			\P\left\{ \max_{1 \le i \le n}\|v(t\,,\eta\,,x_i) - \nu\| \le 2\varepsilon \right\}\\
		& 
			\le \P\left\{ \max_{1 \le i \le n} \left\|\sigma(u_{t-\eta}(x_i))
			\int_{(0, \eta) \times \T} G_{\eta-s}(x_i\,,y) \, \xi^{(t-\eta)}(\d s \d y)
			\right\| \le 2\varepsilon \right\}\\
		&\le\P\left\{  \max_{1\le i\le n} \left\|  \int_{(0, \eta) \times \T}
			G_{\eta-s}(x_i\,,y) \, \xi^{(t-\eta)}(\d s\,\d y) \right\|\le
			2\varepsilon / \inf_{v\in\R^p}\lambda(v) \right\}\\
		&=\P\left\{  \max_{1\le i\le n} \|  H(\eta\,,x_i) \|\le
			2\varepsilon / \inf_{v\in\R^p}\lambda(v) \right\}\\
		&\le \left(\P\left\{  \max_{1\le i\le n} |H_1(\eta\,,x_i)|\le
			2\varepsilon / \inf_{v\in\R^p}\lambda(v) \right\}\right)^p,
	\end{align*}
	where we recall $\lambda(v)$ denotes the smallest singular value of $\sigma(v)$
	and $H$ was defined in \eqref{H} and represents the solution to \eqref{SHE} with $\sigma=$ identity
	matrix and zero initial data. We have also used the facts that: (i) 
	The law of $\xi^{(a)}$ does not
	depend on $a$; and (ii) The coordinates of $H$ are i.i.d. 
	Since the probability density function of a centered, real-valued Gaussian random variable
	$Z$
	is bounded above by $1/\sqrt{2\pi\Var (Z)}$, we proceed iteratively by 
	successive conditioning on the values of $H_1(\eta\,,x_1),\ldots,H_1(\eta\,,x_i)$ 
	as $i$ varies from $n-1$ to 1 in order to see that
	\begin{align*}
		&\P\left\{ \max_{1 \le i \le n}
			\|v(t\,,\eta\,,x_i) - \nu\| \le 2\varepsilon\right\}
			\le \left(\frac{2\varepsilon}{\inf_{v \in \R^p} \lambda(v)
			\sqrt{2\pi \Var\left( H_1(\eta\,,x_1) \right)}}\right)^p\times \\
			&\hskip1.7in\times \prod_{i=2}^n
			\left( \frac{2\varepsilon}{\inf_{v\in\R^p}\lambda(v)\sqrt{2\pi\Var\left(
			H_1(\eta\,,x_i) \mid \mathscr{H}_{i-1}\right)}} \right)^p\\
		& \le C^n \varepsilon^{np} \eta^{-p/4} \prod_{i=2}^n
			\left(  \eta^{p/4} \wedge |x_i - x_{i-1}|^{p/2} \right)^{-1},
	\end{align*}
	where $C>0$ does not depend on the choice of $(t\,,x_1\,,x_2\,,\ldots)$,
	$\mathscr{H}_i$ denotes the $\sigma$-algebra generated
	by $H_1(\eta\,,x_1),\ldots,H_1(\eta\,,x_i)$ when $i=1,\ldots,n-1$,
	and the last inequality follows from Lemma \ref{lem:Var(H)} and 
	strong local nondeterminism (Proposition \ref{pr:LND}).
	Whenever $x_0 \in F_n^\delta$,
	\begin{align*}
		&\sum_{x \in F_n^\delta 
			\setminus \{x_0\}} \left(\eta^{p/4} \wedge |x-x_0|^{p/2}\right)^{-1}
			\le \sum_{1 \le j \le \varepsilon^{-2(1+\delta)}}
			\frac{2}{\eta^{p/4} \wedge (j \varepsilon^{2(1+\delta)})^{p/2}}\\
		&\lesssim \sum_{1 \le j \le \sqrt{\eta} \varepsilon^{-2(1+\delta)}} j^{-p/2}
			\varepsilon^{-p(1+\delta)} + 
			\sum_{\sqrt{\eta} \varepsilon^{-2(1+\delta)} < j \le \varepsilon^{-2(1+\delta)}}
			\eta^{-p/4}\\
		& \lesssim  \begin{cases}
			\varepsilon^{-p(1+\delta)} + \varepsilon^{-2(1+\delta)}
			\eta^{-p/4} & \text{if } p \ge 3,\\
			\varepsilon^{-p(1+\delta)}\log(\sqrt{\eta} \varepsilon^{-2(1+\delta)})
			+ \varepsilon^{-2(1+\delta)} \eta^{-p/4} & \text{if } p = 2,
		\end{cases}
	\end{align*}
	where the implied constants depend only on $p$.
	Now let us suppose $p \ge 4$.
	In that case, we can optimize the preceding bound by choosing
	\begin{align*}
		\eta = \varepsilon^{(1+\delta)(4-(8/p))},
	\end{align*}
	and deduce that
	$\sum_{x \in F^\delta_n \setminus \{x_0\}}
	(\eta^{p/4} \wedge |x-x_0|^{p/2})^{-1} \lesssim \varepsilon^{-p(1+\delta)}$,
	uniformly for all $n\in\N$ and $x_0\in F^\delta_n$.
	It follows that 
	\begin{align*}
		&\E\left[\widetilde{\mathcal{N}}^\delta_n(t\,,\eta\,, B(\nu\,,2\varepsilon))\right] 
			\le \mathop{\sum\cdots\sum}\limits_{x_1< \cdots < x_n\,\text{in}\,F^\delta_n}
			\P\left\{ \max_{1 \le i \le n}\|v(t\,,\eta\,,x_i) - \nu\| \le 2\varepsilon\right\}\\
		&
			\le C_1^n \varepsilon^{pn} \sum_{x_1 \in F^\delta_n}
			\eta^{-p/4}\sum_{x_2 \in F^\delta_n \setminus \{x_1\}}
			\left( \eta^{p/4} \wedge |x_2 - x_1|^{p/2}\right)^{-1} \times\cdots\\
		&\hskip1.6in\cdots\times 
			\sum_{x_n \in F^\delta_n \setminus\{x_{n-1}\}} 
			\left( \eta^{p/4} \wedge |x_n - x_{n-1}|^{p/2}\right)^{-1}\\
		& \le C_2^n \varepsilon^{pn} \varepsilon^{-(1+\delta)pn} = C_2^n \varepsilon^{-\delta pn},
	\end{align*}
	where $C_1,C_2>0$ do not depend on $n$.
	We plug this into \eqref{E:PM}, recall that $\varepsilon = 2^{-n}$, and then appeal
	to Lemma \ref{lem:stirling} in order to see that
	\[
		\P\left\{ {\mathcal{N}}^\delta_n(v\,;t\,,\eta\,, B(\nu\,,2^{-n+1})) \ge 2^{2\delta pn} \right\} 
		\le c^n n^n 2^{-\delta pn^2},
	\]
	where $c$ does not depend on $n$.
	Because $p \ge 4$, we also have $\eta^{1/2} \le \varepsilon^{1+\delta}$.
	Therefore, we may select $k = pn$ in \eqref{E:PN} to find that there exists $c_0,c_1>0$ such that
	\begin{align*}
		&\P\left\{ \mathcal{N}^\delta_n(t\,, B(\nu\,,2^{-n})) \ge 2^{2\delta pn} \right\}\\
		&\le c_0^n n^n 2^{-\delta pn^2} + L_1^{pn} (np)^{3pn/2} \varepsilon^{\delta pn} |\log(1/\eta)|^{3pn/2}
		\le c_1^n n^{3pn} 2^{-\delta pn^2},
	\end{align*}
	uniformly for all $t \in (0\,,T]$, $\nu\in\R^p$ and $n\in\N$.
	This completes the proof of Proposition \ref{pr:key}.
\end{proof}

We are ready to verify Theorem \ref{th:HEAT:torus} and conclude the paper.

\begin{proof}[Proof of Theorem \ref{th:HEAT:torus}]	
	Suppose $p\ge 4$. Choose and fix $0<S<T<\infty$ throughout. It suffices to prove that,
	off a single $\P$-null set,
	\begin{align}\label{E:HEAT:torus}
		\dimh  u(\{t\}\times F) = 2\dimh  F\quad\forall\text{compact
		$F\subset\T,\ t\in[S\,,T]$}.
	\end{align}
	The proof of \eqref{E:HEAT:torus} is divided in two parts.
	In the first part, we verify \eqref{E:HEAT:torus} under the additional
	hypothesis that
	\begin{equation}\label{E:cond}
		\mathcal{M}(b)<\infty\,,\, \mathcal{M}(\sigma)<\infty, \text{ and }
		\inf_{v\in\R^p}\lambda(v)>0.
	\end{equation}
	The second part of the proof is concerned with removing \eqref{E:cond}.\medskip
	
	\emph{Part 1.} Suppose that \eqref{E:cond} holds.
	It is well known that, with probability one, $u$ is H\"older continuous 
	with any fixed index $<\frac12$ in its space variable locally uniformly in time
	and off a single $\P$-null set. 
	More precisely, there exists a $\P$-null set off which
	\begin{equation}\label{E:mod}
		\sup_{t\in[0,T]}\sup_{\substack{x,z\in\T:\\x \neq z}}
		\frac{\| u(t\,,x) - u(t\,,z) \|}{|x-z|^{(1-\varepsilon)/2}}<\infty
		\quad\forall\varepsilon\in(0\,,1),\ T>0;
	\end{equation}
	see for example Lemma \ref{lem:u:mod}.
	Also, thanks to Lemma \ref{lem:u:mod} and Proposition \ref{pr:key}, 
	we can apply Lemma \ref{lem:interpolation} to see that, for 
	every $\delta \in (0\,,1)$ and $R>0$, there exist $K=K(\delta)>0$ 
	and $L=L(p\,,S\,,T\,,b\,,\sigma\,,u_0\,,\delta\,,R)>0$ such that
	\[
		\P\left\{\sup_{t\in [S,T]} \sup_{\nu \in B(0,R)} 
		\mathcal{N}_n^\delta(t\,,B(\nu\,,2^{-n})) \ge K2^{2\delta pn} \right\} 
		\le L^n \e^{-n^2/L}\quad \forall n\in\N.
	\]
	By the Borel--Cantelli lemma, there exists
	a $\P$-null set off which
	\begin{equation}\label{E:N:as}
		\sup_{t\in[S,T]}\max_{\nu\in B(0,R)}
		\mathcal{N}_n(t\,,B(\nu\,,2^{-n})) = \mathcal{O}(2^{2np\delta})\quad\text{as $n\to\infty$}.
	\end{equation}
	With \eqref{E:mod} and \eqref{E:N:as} in place, 
	we can then
	apply Lemma \ref{lem:unif:dim} to obtain \eqref{E:HEAT:torus}.\medskip
	
	\emph{Part 2.} We now apply a truncation argument to
	prove the theorem without assuming \eqref{E:cond}. The truncation argument is somewhat
	delicate and leads to the assumptions of Theorem \ref{th:HEAT:torus}, which are all we
	assume from now on.
	
	Define, for every $N>0$, a function $b_N$ via \eqref{b_N}.
	Recall from \eqref{E:liplip:b} that $b_N$ is globally Lipschitz with
	\begin{equation}\label{E:liplip}
		\lip(b_N) \le \lip(b) \quad\forall N>0.
	\end{equation}
%
	According to Lemma \ref{lem:lambda:cont}, the function $\lambda$ is continuous,
	where $\lambda$ was the minimum singular-value function associated to $\sigma$;
	see Definition \ref{lambda}.
	With \eqref{E:liplip} in place, we begin our truncation
	argument. Define $\Lambda(r):=\lambda^{-1}[0\,,r]=
	\{ v\in\R^p:\, \lambda(v) \le r \}$ to be the level set of $\lambda$ at $r$
	for every $r>0$, and set
	\begin{equation}\label{E:sigma_r}
		\sigma_r(v) = \begin{cases}
			\sigma(v) &\text{if $v\in \Lambda(r)$},\\
			\sigma(v) + \d_{r}(v)\bm{I}&\text{if $v\in\Lambda(r)$},
		\end{cases}
	\end{equation}
	where $\bm{I}=(\delta_0(i-j))_{i,j=1}^p$ denotes the 
	$p\times p$ identity matrix and  $\d_{r}$ is defined to be the ``internal distance function 
	to the restriction of the boundary $\partial \Lambda(r)$''; that is,
	\[
		\d_{r}(x) = \begin{cases}
			\inf\left\{ \|x-y\|:\, y\in\partial\Lambda(r)\right\} & \text{if $x\in\Lambda(r)$},\\
			0&\text{otherwise}.
		\end{cases}
	\]
	The function $\d_{r}$ is Lipschitz continuous whenever $\partial\Lambda(r) \ne \varnothing$ 
	(see Lemma \ref{lem:d_Lip} below), and this is the case when $r > 0$ is sufficiently small.
	In this case, $\sigma_{r}$ is Lipschitz continuous.
	Next, set
	\begin{equation}\label{E:sigma_{r,N}}
		\sigma_{r,N}(v) = \begin{cases}
			\sigma_r(v) &\text{if $\|v\|\le N$},\\
			\sigma_r(vN/\|v\|)&\text{if $\|v\|>N$}.
		\end{cases}
	\end{equation}
	As in \eqref{E:liplip:b},
	$\lip(\sigma_{r,N}) \le \lip(\sigma_r)$ for all $r,N>0.$
	Therefore, $\sigma_{r,N}$ is bounded and Lipschitz continuous provided that $r \ll 1$ and $N \gg 1$.
		
	Let $\lambda(v\,;r)$ and $\lambda(v\,;N\,,r)$ respectively
	denote the smallest singular value of $\sigma_r$ and $\sigma_{r,N}$. That is,
	\[
		\lambda(v\,;r) =\inf_{x\in\R^p:\|x\|=1}\|\sigma_{r}(v)x\|^2,\qquad
		\lambda(v\,;N\,,r) =\inf_{x\in\R^p:\|x\|=1}\|\sigma_{r,N}(v)x\|^2.
	\]
	Thanks to \eqref{E:sigma_r} and \eqref{E:sigma_{r,N}},
	\[
		\lambda(v\,;N\,,r) = \begin{cases}
			\lambda(v) &\text{if $v\in B(0\,,N)\setminus\Lambda(r)$},\\
			\lambda(v) + \d_{r}(v)^2&\text{if $v\in B(0\,,N)\cap\Lambda(r)$},\\
			\lambda(vN/\|v\|\,;r)&\text{if $\|v\|>N$}.
		\end{cases}
	\]
	Apply Lemma \ref{lem:lambda:cont} with $\sigma$ replaced by $\sigma_{r,N}$
	to see that  $v\mapsto \lambda(v\,;N\,,r)$ is  continuous. By virtue of its
	construction, $\lambda(v\,;N\,,r)>0$ everywhere in $B(0\,,N)$. Since
	$\inf_{\|v\|>N}\lambda(v\,;N\,,r)=\inf_{\|v\|=N}\lambda(v\,;N\,,r)$,
	compactness yields
	\begin{equation}\label{E:lll}
		\inf_{v\in\R^p}\lambda(v\,;N\,,r) = \inf_{v\in B(0,N)}\lambda(v\,;N\,,r)>0.
	\end{equation}
	With the above observations in place, let us write $u_{r,N}$ for the solution to
	\eqref{SHE} where $b$ is replaced by $b_N$ and $\sigma$ is replaced by $\sigma_{r,N}$. 
	Since $\mathcal{M}(b_N) = \sup_{\|v\|\le N}\|b(v)\| < \infty$ and 
	$\mathcal{M}(\sigma_{r,N})=\sup_{\|v\|\le N}\|\sigma_r(v)\|<\infty$, and because of
	\eqref{E:lll}, we may apply Part 1 of this proof to $u_{r,N}$ in place of $u$ to see that
	\[
		\P\left\{ \dimh u_{r,N}(\{t\}\times F) = 2\dimh F
		\ \forall\text{compact  $F\subset\T,\ t>0$}\right\}=1.
	\]
	Let
	\[
		T_{r,N}=\inf\left\{ t>0:\, \sup_{x\in\T}\|u(t\,,x)\|\ge N \right\}
		\wedge\inf\left\{t>0:\, \lambda(u(t\,,x)) \le r\right\},
	\]
	where $\inf\varnothing=\infty$.
	Every $T_{r,N}$ is a stopping time
	with respect to the filtration $\{\sF_t\}_{t\ge0}$ of the noise $\xi$, and \eqref{E:sigma_{r,N}}
	and the uniqueness of the solution to \eqref{SHE} together imply that
	\[
		T_{r,N}=\inf\left\{ t>0:\, \sup_{x\in\T}\|u_{r,N}(t\,,x)\|\ge N \right\}
		\wedge\inf\left\{t>0:\, \lambda(u_{r,N}(t\,,x)) \le r\right\},
	\]
	and $\P\{ u_{r,N}(t)=u(t)\text{ for all $t<T_{r,N}$}\}=1$. Also,
	off a single $\P$-null set,
	\begin{equation}\label{E:end}
		\dimh u(\{t\}\times F) = 2\dimh F\quad \forall
		\text{compact $F\subset\T,\ t\in(0\,,T_{r,N})$}.
	\end{equation}
	Because $u$ is bounded on space-time compacta and $\{\lambda=0\}$ is polar for $u$
	[Assumption \ref{ass:polar}],
	$\lim_{N\to\infty,r\to0}\P\{T_{r,N}>t\}=1$ for every $t>0$.
	Therefore, \eqref{E:end} implies the result and 
	concludes the proof.
\end{proof}

\appendix 
\section{A miscellany of related results}

This appendix contains a few technical results that are used in the body of the paper.
The following is well known. We include a short proof for the sake of completeness.

\begin{lemma}\label{lem:d_Lip}
Let $A$ be a nonempty subset of $\R^p$ and $d_A: \R^p \to \R$ 
be the distance function defined by $d_A(x) = \inf\{ \|x-z\| : z \in A \}$.
Then $d_A$ is Lipchitz continuous function with Lipschitz constant $1$.
\end{lemma}

\begin{proof}
	If $x, y \in \R^p$, then the triangle inequality yields
	$d_A(x) \le \|x-z\| \le \|x-y\| + \|y-z\|$ for all $z \in A$. 
	Take infimum over $z \in A$ to see that $d_A(x) - d_A(y) \le \|x-y\|$.
 	Interchange the roles of $x$ and $y$ to conclude.
\end{proof}

The following simple fact is found by splitting the sum according to whether or
not $n>\lambda^{-1/p}$, particularly relevant only when $\lambda\in(0\,,1)$.

\begin{lemma}\label{inequality1}
	If $p>1$, then $\sum_{n=1}^\infty n^{-2}\min(1\,,\lambda n^p)\lesssim
	\lambda^{1/p}$ uniformly for all $\lambda>0$.
\end{lemma}


The following is a basic form
of the Poisson summation formula \cite{ShS03}.

\begin{lemma}\label{lem:poisson}
	$\sum_{n=-\infty}^\infty f(n) = \sum_{n=-\infty}^\infty \hat{f}(2\pi n)$
	$\forall f\in\mathscr{S}(\R)$.
\end{lemma}

The following is a multidimensional version of the Burkholder-Davis-Gundy
inequality for stochastic convolutions. When $p=1$, this inequality can be found for example
in  \cite{DK}*{Proposition 4.4}. The proof in the general case follows a similar route.
We include it since the precise constants below might require some justification.

\begin{lemma}[BDG inequality]\label{lem:BDG}
	Whenever $Z=\{Z(s\,,y)\}_{s>0,y\in\T}$ is a predictable random field
	with values in the space of $p\times p$ matrices, 
	\[
		\left\| \int_{(0,t)\times\T} G_{t-s}(x\,,y) Z(s\,,y)\,\xi(\d s\,\d y)\right\|_k^2
		\le 4kp\int_0^t\d s\int_{\T}\d y\ [G_{t-s}(x\,,y)]^2
		\left\| Z(s\,,y)\right\|_k^2,
	\]
	for every $t>0$, $x\in\T$, and $k\in[2\,,\infty)$.
\end{lemma}

\begin{proof}
	There is nothing to prove when $\int_0^t\d s\int_{\T}\d y\ [G_{t-s}(x\,,y)]^2
	\| Z(s\,,y)\|_k^2$ is infinite. Therefore, we assume throughout that the integral is finite. Because
	of that fact it follows that
	\[
		M_0=0,
		\quad
		M_t = \int_{(0,t)\times\T} G_{T-s}(x\,,y) \sum_{j=1}^pZ_{i,j}(s\,,y)\,\xi_j(\d s\,\d y)
		\qquad[0<t\le T]
	\]
	defines a continuous $L^2$-martingale for every $T>0$.
	If $X$ is random variable with values in $\R^p$, then
	\begin{equation}\label{||X||}
		\|X\|_k^2  = \left[ \E\left( \|X\|^k\right) \right]^{2/k}
		\le\sum_{i=1}^p \left\|  |X_i|^2\right\|_{k/2} = \sum_{i=1}^p \|X_i\|_k^2.
	\end{equation}
	For every $i=1,\ldots,p$, the quadratic variation of $\{M_t\}_{t\in[0,T]}$
	is 
	\[
		\<M\>_t= \int_0^t\d s\int_{\T}\d y\ [G_{T-s}(x\,,y)]^2 \sum_{j=1}^p |Z_{i,j}(s\,,y)|^2
		\qquad \forall t\in(0\,,T].
	\]
	Therefore, we apply \eqref{||X||} and the BDG inequality for stochastic convolutions
	\cite{DK}{Prop.\ 4.4} with $t=T$ to see that
	\begin{align*}
		&\left\| \int_{(0,t)\times\T} G_{t-s}(x\,,y) Z(s\,,y)\,\xi(\d s\,\d y)\right\|_k^2\\
		&\hskip1in\le \sum_{i=1}^p\left\| 
			\int_{(0,t)\times\T} G_{t-s}(x\,,y) 
			\sum_{j=1}^pZ_{i,j}(s\,,y)\,\xi_j(\d s\,\d y)\right\|_k^2\\
		&\hskip1in\le4k\sum_{i=1}^p \left\|\int_0^t\d s\int_{\T}\d y\ [G_{t-s}(x\,,y)]^2
			\sum_{j=1}^p |Z_{i,j}(s\,,y)|^2\right\|_{k/2}\\
		&\hskip1in\le4k\sum_{i=1}^p \int_0^t\d s\int_{\T}\d y\ [G_{t-s}(x\,,y)]^2
			\left\|\sum_{j=1}^p |Z_{i,j}(s\,,y)|^2\right\|_{k/2}\\
		&\hskip1in\le4kp \int_0^t\d s\int_{\T}\d y\ [G_{t-s}(x\,,y)]^2
			\left\| \|Z(s\,,y)\|^2\right\|_{k/2}.
	\end{align*}
	This implies the lemma.
\end{proof}

We will need the following particular application of Stirling's formula.

\begin{lemma}\label{lem:stirling}
	For every $\alpha>0$,
	$\binom{\lceil 2^{\alpha n}\rceil}{n} \sim  
	(2\pi n)^{-1/2} 2^{\alpha n^2}
	(\e/n)^n$ as $n\to\infty$.
\end{lemma}

And the following is a well-known consequence of a direct covering argument; see for example
Falconer \cite{Fa}*{Proposition 2.3}.

\begin{lemma}\label{lem:easy}
	If $\exists \alpha>0$
	such that $f:\T\to\R^p$ satisfies
	$\|f(x)-f(z)\|\lesssim |x-z|^\alpha$ uniformly for all $x,z\in\T$,
	then $\dimh f(F)\le \alpha^{-1}\dimh F$ for every 
	Borel set $F\subset\T$.
\end{lemma}

%

Finally, the following is an infinite-dimensional extension of the well-known
Girsanov theorem. See Allouba \cite{Allouba} for a proof in the case
that $p=1$ and Da Prato and Zabczyk \cite{DZ} for a quite general,
abstract version.

\begin{lemma}[Girsanov's theorem]\label{lem:girsanov}
	Choose and fix a number $T>0$, and a 
	predictable random field $\{Z(t\,,x)\}_{(t,x)\in[0,T]\times\T}$, with values in $\R^p$,
	that satisfies
	$\E \exp( \frac12  \|Z\|_{L^2([0,T]\times\T)}^2 ) <\infty$. Then,
	$\zeta(\d t\,\d x) = Z(t\,,x)\d t\,\d x + \xi(\d t\,\d x)$
	is  a $p$-dimensional space-time white noise on $(\Omega\,, \sF_T, \rQ)$, 
	where
	\[
		\d\rQ/\d\P = \exp\left(-M_T-\tfrac12\<M\>_T\right),
	\]
	for $M_t=\int_{[0, t]\times \T}  Z \cdot \d\xi$ for every $t\in[0\,,T]$.
\end{lemma}

\bibliographystyle{plain}
\bibliography{UniformDimension}

\begin{bibdiv}
\begin{biblist}

\bib{Allouba}{article}{
      author={Allouba, H.},
       title={Different types of {SPDE}s in the eyes of {G}irsanov's theorem},
        date={1998},
        ISSN={0736-2994,1532-9356},
     journal={Stochastic Anal. Appl.},
      volume={16},
      number={5},
       pages={787\ndash 810},
         url={https://doi.org/10.1080/07362999808809562},
      review={\MR{1643116}},
}

\bib{An}{article}{
      author={Anderson, T.~W.},
       title={The integral of a symmetric unimodal function over a symmetric
  convex set and some probability inequalities},
        date={1955},
        ISSN={0002-9939,1088-6826},
     journal={Proc. Amer. Math. Soc.},
      volume={6},
       pages={170\ndash 176},
         url={https://doi.org/10.2307/2032333},
      review={\MR{69229}},
}

\bib{B73}{article}{
      author={Berman, Simeon~M.},
       title={Local nondeterminism and local times of {G}aussian processes},
        date={1973/74},
        ISSN={0022-2518,1943-5258},
     journal={Indiana Univ. Math. J.},
      volume={23},
       pages={69\ndash 94},
         url={https://doi.org/10.1512/iumj.1973.23.23006},
      review={\MR{317397}},
}

\bib{DZ}{book}{
      author={Da~Prato, Giuseppe},
      author={Zabczyk, Jerzy},
       title={Stochastic {E}quations in {I}nfinite {D}imensions},
     edition={Second},
      series={Encyclopedia of Mathematics and its Applications},
   publisher={Cambridge University Press, Cambridge},
        date={2014},
      volume={152},
        ISBN={978-1-107-05584-1},
         url={https://doi.org/10.1017/CBO9781107295513},
      review={\MR{3236753}},
}

\bib{DKN1}{article}{
      author={Dalang, Robert~C.},
      author={Khoshnevisan, Davar},
      author={Nualart, Eulalia},
       title={Hitting probabilities for systems of non-linear stochastic heat
  equations with additive noise},
        date={2007},
        ISSN={1980-0436},
     journal={ALEA Lat. Am. J. Probab. Math. Stat.},
      volume={3},
       pages={231\ndash 271},
      review={\MR{2365643}},
}

\bib{DKN2}{article}{
      author={Dalang, Robert~C.},
      author={Khoshnevisan, Davar},
      author={Nualart, Eulalia},
       title={Hitting probabilities for systems for non-linear stochastic heat
  equations with multiplicative noise},
        date={2009},
        ISSN={0178-8051,1432-2064},
     journal={Probab. Theory Related Fields},
      volume={144},
      number={3-4},
       pages={371\ndash 427},
         url={https://doi.org/10.1007/s00440-008-0150-1},
      review={\MR{2496438}},
}

\bib{Dudley}{article}{
      author={Dudley, R.~M.},
       title={The sizes of compact subsets of {H}ilbert space and continuity of
  {G}aussian processes},
        date={1967},
        ISSN={0022-1236},
     journal={J. Functional Analysis},
      volume={1},
       pages={290\ndash 330},
         url={https://doi.org/10.1016/0022-1236(67)90017-1},
      review={\MR{220340}},
}

\bib{Fa}{book}{
      author={Falconer, Kenneth},
       title={Fractal {G}eometry},
     edition={Third},
   publisher={John Wiley \& Sons, Ltd., Chichester},
        date={2014},
        ISBN={978-1-119-94239-9},
        note={Mathematical foundations and applications},
      review={\MR{3236784}},
}

\bib{Fife}{book}{
      author={Fife, Paul~C.},
       title={Mathematical {A}spects of {R}eacting and {D}iffusing {S}ystems},
      series={Lecture Notes in Biomathematics},
   publisher={Springer-Verlag, Berlin-New York},
        date={1979},
      volume={28},
        ISBN={3-540-09117-3},
      review={\MR{527914}},
}

\bib{FKM}{article}{
      author={Foondun, Mohammud},
      author={Khoshnevisan, Davar},
      author={Mahboubi, Pejman},
       title={Analysis of the gradient of the solution to a stochastic heat
  equation via fractional {B}rownian motion},
        date={2015},
        ISSN={2194-0401,2194-041X},
     journal={Stoch. Partial Differ. Equ. Anal. Comput.},
      volume={3},
      number={2},
       pages={133\ndash 158},
         url={https://doi.org/10.1007/s40072-015-0045-y},
      review={\MR{3350450}},
}

\bib{HL}{misc}{
      author={Hu, Jingwu},
      author={Lee, Cheuk~Yin},
       title={On the spatio-temporal increments of nonlinear parabolic {SPDE}s
  and the open {KPZ} equation},
        date={2025},
        note={Preprint available at arXiv:2508.05032},
}

\bib{Kaufman}{article}{
      author={Kaufman, Robert},
       title={Une propri\'et\'e{} m\'etrique du mouvement brownien},
        date={1969},
        ISSN={0151-0509},
     journal={C.\ R.\ Acad. Sci. Paris S\'er. A-B},
      volume={268},
       pages={A727\ndash A728},
      review={\MR{240874}},
}

\bib{DK}{book}{
      author={Khoshnevisan, Davar},
       title={Analysis of {S}tochastic {P}artial {D}ifferential {E}quations},
      series={CBMS Regional Conference Series in Mathematics},
   publisher={Conference Board of the Mathematical Sciences, Washington, DC; by
  the American Mathematical Society, Providence, RI},
        date={2014},
      volume={119},
        ISBN={978-1-4704-1547-1},
         url={https://doi.org/10.1090/cbms/119},
      review={\MR{3222416}},
}

\bib{KKMS20}{article}{
      author={Khoshnevisan, Davar},
      author={Kim, Kunwoo},
      author={Mueller, Carl},
      author={Shiu, Shang-Yuan},
       title={Dissipation in parabolic {SPDE}s},
        date={2020},
        ISSN={0022-4715,1572-9613},
     journal={J. Stat. Phys.},
      volume={179},
      number={2},
       pages={502\ndash 534},
         url={https://doi.org/10.1007/s10955-020-02540-0},
      review={\MR{4091567}},
}

\bib{KSXZ}{misc}{
      author={Khoshnevisan, Davar},
      author={Swanson, Jason},
      author={Xiao, Yimin},
      author={Zhang, Liang},
       title={Weak existence of a solution to a differential equation driven by
  a very rough fbm},
        date={2013},
        note={Unpublished manuscript, available at arXiv:1309.3613},
}

\bib{LX}{misc}{
      author={Lee, Cheuk~Yin},
      author={Xiao, Yimin},
       title={Local times of anisotropic {G}aussian random fields and
  stochastic heat equation},
        date={2023},
        note={Preprint available at arXiv:2308.13732},
}

\bib{McKean}{article}{
      author={McKean, Henry~P., Jr.},
       title={Hausdorff-{B}esicovitch dimension of {B}rownian motion paths},
        date={1955},
        ISSN={0012-7094,1547-7398},
     journal={Duke Math. J.},
      volume={22},
       pages={229\ndash 234},
         url={http://projecteuclid.org/euclid.dmj/1077466268},
      review={\MR{69425}},
}

\bib{MP87}{incollection}{
      author={Monrad, Ditlev},
      author={Pitt, Loren~D.},
       title={Local nondeterminism and {H}ausdorff dimension},
        date={1987},
   booktitle={Seminar on {S}tochastic {P}rocesses, 1986 ({C}harlottesville,
  {V}a., 1986)},
      series={Progr. Probab. Statist.},
      volume={13},
   publisher={Birkh\"auser Boston, Boston, MA},
       pages={163\ndash 189},
      review={\MR{902433}},
}

\bib{Pitt78}{article}{
      author={Pitt, Loren~D.},
       title={Local times for {G}aussian vector fields},
        date={1978},
        ISSN={0022-2518,1943-5258},
     journal={Indiana Univ. Math. J.},
      volume={27},
      number={2},
       pages={309\ndash 330},
         url={https://doi.org/10.1512/iumj.1978.27.27024},
      review={\MR{471055}},
}

\bib{ShS03}{book}{
      author={Stein, Elias~M.},
      author={Shakarchi, Rami},
       title={Fourier {A}nalysis},
      series={Princeton Lectures in Analysis},
   publisher={Princeton University Press, Princeton, NJ},
        date={2003},
      volume={1},
        ISBN={0-691-11384-X},
      review={\MR{1970295}},
}

\bib{Turing}{article}{
      author={Turing, A.~M.},
       title={The chemical basis of morphogenesis},
        date={1952},
        ISSN={0080-4622},
     journal={Philos. Trans. Roy. Soc. London Ser. B},
      volume={237},
      number={641},
       pages={37\ndash 72},
      review={\MR{3363444}},
}

\bib{W}{incollection}{
      author={Walsh, John~B.},
       title={An {I}ntroduction to {S}tochastic {P}artial {D}ifferential
  {E}quations},
        date={1986},
   booktitle={\'ecole d'\'et\'e{} de probabilit\'es de {S}aint-{F}lour,
  {XIV}---1984},
      series={Lecture Notes in Math.},
      volume={1180},
   publisher={Springer, Berlin},
       pages={265\ndash 439},
         url={https://doi.org/10.1007/BFb0074920},
      review={\MR{876085}},
}

\bib{WuXiao}{article}{
      author={Wu, Dongsheng},
      author={Xiao, Yimin},
       title={Uniform dimension results for {G}aussian random fields},
        date={2009},
        ISSN={1006-9283,1862-2763},
     journal={Sci. China Ser. A},
      volume={52},
      number={7},
       pages={1478\ndash 1496},
         url={https://doi.org/10.1007/s11425-009-0103-x},
      review={\MR{2520589}},
}

\bib{X09}{incollection}{
      author={Xiao, Yimin},
       title={Sample path properties of anisotropic {G}aussian random fields},
        date={2009},
   booktitle={A {M}inicourse on {S}tochastic {P}artial {D}ifferential
  {E}quations},
      series={Lecture Notes in Math.},
      volume={1962},
   publisher={Springer, Berlin},
       pages={145\ndash 212},
         url={https://doi.org/10.1007/978-3-540-85994-9_5},
      review={\MR{2508776}},
}

\end{biblist}
\end{bibdiv}

\end{document}